\documentclass[10pt]{amsart}
\usepackage{amsmath,amscd}
\usepackage{graphicx}
\usepackage[all]{xy}
\usepackage{amsthm}
\usepackage{amssymb,url}
\usepackage[charter]{mathdesign}
\usepackage{amsfonts}
\usepackage{mathrsfs}
\usepackage{amsxtra}
\usepackage{bbm}
\usepackage{ae}
\usepackage[all]{xy}
\usepackage{color}
\usepackage{subfigure}
\usepackage{url}

\newtheorem{thm}{Theorem}
\newtheorem{thm*}{Theorem}
\newtheorem{lem}[thm]{Lemma}
\newtheorem{cor}[thm]{Corollary}
\newtheorem{conj}[thm]{Conjecture}
\newtheorem{prop}[thm]{Proposition}

\theoremstyle{remark}
\newtheorem{remark}[thm]{Remark}
\newtheorem{example}[thm]{Example}

\theoremstyle{definition}
\newtheorem{deef}[thm]{Definition}

\newcommand{\R}{\mathbbm{R}}
\newcommand{\C}{\mathbbm{C}}

\newcommand{\Z}{\mathbbm{Z}}

\newcommand{\N}{\mathbbm{N}}

\newcommand{\rd}{\mathrm{d}}

\newcommand{\grad}{\mathrm{g} \mathrm{r} \mathrm{a} \mathrm{d} \,}

\renewcommand{\epsilon}{\varepsilon}

\newcommand{\im}{\mathrm{i} \mathrm{m} \,}

\newcommand{\e}{\mathrm{e}}

\title[Magnitude function of domains in Euclidean space]{On the magnitude function of domains
in Euclidean space}
\author{Heiko Gimperlein, Magnus Goffeng}
\address{Heiko Gimperlein,\newline
\indent Maxwell Institute for Mathematical Sciences and \newline
\indent Department of Mathematics, Heriot-Watt University\newline
\indent Edinburgh EH14 4AS\newline
\indent United Kingdom\newline
\newline
\indent {\it and}\newline
\newline 
\indent Institute for Mathematics, University of Paderborn\newline
\indent Warburger Str.~100\newline
\indent 33098 Paderborn\newline
\indent Germany\newline
\newline
\indent Magnus Goffeng,\newline
\indent Department of Mathematical Sciences\newline 
\indent Chalmers University of Technology and \newline 
\indent University of Gothenburg\newline 
\indent SE-412 96 Gothenburg\newline 
\indent Sweden\newline}
\subjclass[2010]{}
\keywords{51F99 (primary), 28A75, 58J40, 58J50 (secondary)}
\email{h.gimperlein@hw.ac.uk, goffeng@chalmers.se}

\begin{document}
\maketitle

\begin{abstract}
We study Leinster's notion of magnitude for a compact metric space. For a smooth, compact domain $X\subset \mathbb{R}^{2m-1}$, we find geometric significance in the  function $\mathcal{M}_X(R) = \mathrm{mag}(R\cdot X)$. The function $\mathcal{M}_X$ extends from the positive half-line to a meromorphic function in the complex plane. Its poles are generalized scattering resonances. In the semiclassical limit $R \to \infty$, $\mathcal{M}_X$ admits an asymptotic expansion. The three leading terms of $\mathcal{M}_X$ at $R=+\infty$ are proportional to the volume, surface area and integral of the mean curvature. In particular, for convex $X$ the leading terms are proportional to the intrinsic volumes, and we obtain an asymptotic variant of the convex magnitude conjecture by Leinster and Willerton, with corrected coefficients. 
\end{abstract}

\large
\section{Introduction}
\normalsize

The notion of \emph{magnitude} of an enriched category, and specifically for a compact metric space, was introduced by Leinster \cite{leinster} to capture the ``essential size'' of an object. The magnitude has been shown to generalize the cardinality of a set and the Euler characteristic, and it is even closely related to measures of the diversity of a biological system. See \cite{leinmeck} for an overview.

Let $(X,d)$ be a finite metric space. We say that $w : X \to \R$ is a weight function provided that $\sum_{y \in X} \mathrm{e}^{-\rd(x,y)}w(y) =1$ for all $x \in X$. Given a weight function $w$, the magnitude of $X$ is defined as $\mathrm{mag}(X) := \sum_{x \in X} w(x)$. It is easy to check that $\mathrm{mag}(X)$ is independent of the choice of $w$, and a weight function exists when $(X,\rd)$ is positive definite (i.e.~ the matrix $(\mathrm{e}^{-\rd(x,y)})_{x,y\in F}$ is positive definite for any finite $F\subseteq X$). The magnitude of a compact, positive definite metric space $(X,\rd)$ is defined as 
$$\mathrm{mag}(X) := \sup\{\mathrm{mag}(\Xi) : \Xi \subset X\ \ \text{finite}\}\ .$$

Few explicit computations of magnitude are known. We provide a framework for a refined analysis and explicit computations when $X \subset \mathbb{R}^n$ is smooth, compact and $n=2m-1$ is odd. The framework builds on classical techniques from spectral geometry and semiclassical analysis and applies them to a mixed order system associated with a boundary value problem of order $n+1$. As an application, we find a geometric origin of the magnitude and prove an asymptotic variant of the Leinster-Willerton conjecture.\\

In the case of compact sets $X \subset \mathbb{R}^n$, Meckes \cite{meckes} gives an interpretation of the magnitude in potential theoretic terms, as a generalized capacity:
\begin{equation}\label{magsobolev}
\mathrm{mag}(X) = \frac{1}{n! \omega_n} \inf \left\{\|(1-\Delta)^{(n+1)/4} f\|_{L^2(\mathbb{R}^n)}^2 : f \in H^{(n+1)/2}(\mathbb{R}^n),\ f =1 \ \text{on}\ X\right\}\ .
\end{equation}
Here $\omega_n$ denotes the volume of the unit ball in $\R^n$ and $H^{(n+1)/2}(\mathbb{R}^n):=(1-\Delta)^{-(n+1)/4}L^2(\R^n)$ the Sobolev space of exponent $(n+1)/2$.

Instead of the magnitude of an individual set $X$, it proves fruitful to study the function $\mathcal{M}_X(R) := \mathrm{mag}(R\cdot X)$ for $R>0$. Here, for $X\subset \mathbb{R}^n$ we use the notation $R\cdot X = \{Rx:x\in X\}\subseteq \R^n$. Motivated by properties of the Euler characteristic as well as heuristics and computer calculations, Leinster and Willerton \cite{leinwill} conjectured a relation to the intrinsic volumes $V_i(X)$:

\begin{conj}
\label{lwconj}
Suppose $X\subset \mathbb{R}^n$ is compact and convex. Then 
$$\mathcal{M}_X(R) = \frac{1}{n! \omega_n} \mathrm{vol}_n(X)\ R^n +  \frac{1}{2(n-1)! \omega_{n-1}} \mathrm{vol}_{n-1}(\partial X)\ R^{n-1} + \cdots + 1 = \sum_{i=0}^n\frac{1}{i! \omega_i} V_i(X)\ R^i \ .$$
\end{conj}

In particular, this formula in terms of the intrinsic volumes $V_i(X)$ would imply continuity properties and an inclusion-exclusion principle for the magnitude. We refer to \cite{barcarbs, leinwill} for further motivation. 

Explicitly computing the infimum in formula \eqref{magsobolev} for the unit ball $B_5$ in dimension $5$ using separation of variables, Barcel\'o and Carbery \cite{barcarbs} disproved Conjecture \ref{lwconj}: In this case they found that $\mathcal{M}_{B_5}$ is a rational function with one pole located in $R=-3$. \\

In spite of this negative result, we here describe the geometric content in the magnitude function $\mathcal{M}_X$. The main application of our approach is an asymptotic variant of  Conjecture \ref{lwconj}, also for non-convex domains. \\

To state the main result, recall that a domain $X\subseteq \R^n$ is a subset which coincides with the closure of its interior points. A domain is called smooth if its boundary $\partial X$ is a smooth hypersurface in $\R^n$. A sector $\Gamma\subseteq \C$ is an open subset closed under multiplication by $t\geq 1$. For a function $f:\Gamma\to \C$ and numbers $c_j\in \C$, $j=0,1,\ldots$, we write $f(R)\ \sim \ \sum_{j=0}^\infty c_j R^{n-j}$ as $\mathrm{Re}(R)\to +\infty$ if for any $M\in \N$, there is an $N\in \N$ such that 
$$f(R)-\sum_{j=0}^N c_jR^{n-j} =\mathcal{O}(|R|^{-M}) \quad\mbox{as $\mathrm{Re}(R)\to +\infty$ in $\Gamma$}.$$
{For $\alpha\in [0,\pi)$ we use the notation $\Gamma_\alpha := \{ R \in \C\setminus \{0\}: |\mathrm{arg}(R)|<\alpha\}$.}

\begin{thm}
\label{mainthm}
Let $X\subset \R^{n}$ be a smooth, compact domain, where $n=2m-1$ is odd. 
\begin{enumerate}
\item[a)] $\mathcal{M}_X$ admits a meromorphic continuation to $\mathbb{C}$.
{ \item[b)] For any $\alpha\in [0,\pi/2)$, there is a constant $C_{\alpha,X}\geq 0$, such that $\mathcal{M}_X$ is holomorphic in $\Gamma_{\frac{\pi}{n+1}}\cup \{R\in \Gamma_{\alpha}: \mathrm{Re}(R)\geq C_{\alpha,X}\}$. The poles of $\mathcal{M}_X$ are contained in a set of generalized scattering resonances $\mathrm{P}_{\Lambda}(X)\setminus\{0\}$  (see Definition \ref{scattpo}). }
\item[c)] Let $\Gamma$ be a sector in $\C_+ := \{R \in \C : \mathrm{Re}(R)>0\}$. There exists an asymptotic expansion
\begin{equation}
\label{mxas}
\mathcal{M}_X(R)\ \sim \ \frac{1}{n!\omega_n}\sum_{j=0}^\infty c_j R^{n-j} \quad \mbox{as $\mathrm{Re}(R)\to +\infty$ in $\Gamma$}
\end{equation}
with coefficients $c_j=c_j(X)$, $j=0,1,2,\ldots$
\item[d)] The first three coefficients are given by
$$ c_0(X)=\textnormal{vol}_n(X),\  c_1(X)=m\textnormal{vol}_{n-1}(\partial X),\  c_2(X)=\frac{m^2}{2}\ (n-1)\int_{\partial X} H\, \rd S\ .$$
Here $H$ is the mean curvature of $\partial X$. If $X$ is convex, $c_j$ is proportional to the intrinsic volume $V_{n-j}(X)$, $j=0,1,2$.
\item[e)] For $j\geq 1$, the coefficient $c_{j}$ is determined by the second fundamental form $L$ and covariant derivative $\nabla_{\partial X}$ of $\partial X$: $c_j$ is an integral over $\partial X$ of a universal polynomial in the entries of $\nabla_{\partial X}^k L$, $0 \leq k\leq j-2$. The total number of covariant derivatives appearing in each term of the polynomial is $j-2$.
\end{enumerate}
\end{thm}

\begin{remark}
Part e) is deduced from a statement in local coordinates using invariance theory, as in Atiyah, Bott and Patodi's work \cite{abp}, and Gilkey \cite{gilkey}. It implies, in particular, that for smooth, compact domains $A$, $B$ and $A \cap B \subset \R^n$, the inclusion-exclusion principle $\mathcal{M}_{A \cup B}(R) \sim \mathcal{M}_A(R) + \mathcal{M}_B(R) - \mathcal{M}_{A \cap B}(R)$ holds.
\end{remark}

\begin{remark}
The pseudo-differential techniques in this article still apply when $\partial X$ is not smooth, but only of class $C^r$ for a sufficiently large $r$. The relevant non-smooth calculus of pseudodifferential operators was introduced by Kumano-go and Nagase and studied, in particular, by Marschall in \cite{marschall}; recent work shows its invariance under changes of variables and thereby allows applications to operators on nonsmooth manifolds \cite{abels}.  As a consequence, $\mathcal{M}_X$ still admits a meromorphic continuation to $\C$, and  $\mathcal{M}_X(R)=\frac{1}{n!\omega_n}\sum_{j=0}^N c_j R^{n-j} + o(R^{n-N})$ as $\mathrm{Re}(R) \to +\infty$ in any sector in $\C_+$, with $N = N(r)$. As before, one obtains an asymptotic inclusion-exclusion principle $\mathcal{M}_{A \cup B}(R) = \mathcal{M}_A(R) + \mathcal{M}_B(R) - \mathcal{M}_{A \cap B}(R)+ o(R^{n-N})$ provided $A$, $B$ and $A\cap B  \subset \R^n$ are compact $C^r$ domains.
\end{remark}

\begin{remark}
The leading term $c_0$ in the asymptotics \eqref{mxas} was known from \cite{barcarbs}, and some formulas for $\mathcal{M}_X$ have been obtained for balls \cite{will}. More precisely, Willerton \cite{will} obtains formulas for the magnitude of the $(2m-1)$-ball which imply the asymptotics 
\begin{equation}
\label{willertsasym}
n! \mathcal{M}_{B_n}(R) = R^n + \frac{n (n+1)}{2} R^{n-1} + \frac{n m^2 (n-1)}{2} R^{n-2} + \mathcal{O}(R^{n-3}).
\end{equation} 
The asymptotics \eqref{willertsasym} proves part d) of Theorem \ref{mainthm} for the unit ball. We use this fact to simplify a combinatorial argument in our proof for general $X$ in Subsection \ref{computingcoeffsubs}. { In Example \ref{exampleShell} we apply the approach of \cite{barcarbs, will} to the spherical shell $X=(2B_3)\setminus B_3^\circ$. The magnitude function $\mathcal{M}_X$ is not rational and contains an infinite sequence of poles with unbounded real parts, accumulating on the curve $\mathrm{Re}(R)= \log(|\mathrm{Im}(R)|)$}.
\end{remark}

\begin{remark}
The tools used in the proof of Theorem \ref{mainthm} are from global analysis, and are often seen in the related problems of computing the semiclassical behavior of resolvents, asymptotics of heat traces and spectral asymptotics. Historically, the geometric content of heat trace asymptotics was crucial for Patodi's approach to the Atiyah-Singer index theorem, see for instance \cite{abp}.
\end{remark}

\begin{remark}
Since the submission of this paper, we have computed the fourth term $c_3(X)$ to be proportional to $\int_{\partial X} H^2\rd S$. Details will be published elsewhere. This term is not proportional to an intrinsic volume, e.g. in dimension $n=3$ the expression $\int_{\partial X} H^2\rd S$ computes the Willmore energy of the $\partial X$ whereas $V_0(X)$ is proportional to the Euler characteristics of $X$. In particular, the Leinster-Willerton conjecture \cite{leinwill} can only hold up to an error term $O(R^{n-3})$ for smooth domains. Moreover, the fact that $\int_{\partial X} H^2\rd S$ can be infinite for boundaries with less regularity than $C^2$ opens up for potentially exotic asymptotic behaviour of the magnitude function of low regularity domains.
\end{remark}

\subsection*{Outlook} 
The framework in this article suggests an approach to natural problems:
\begin{enumerate}
\item[A)] Count the number of poles of $\mathcal{M}_X$.
\item[B)] In even dimension, prove meromorphic continuation of the magnitude function and an asymptotic Leinster-Willerton conjecture.
\item[C)]  Study the magnitude function $R \mapsto \mathrm{mag}(X,R \cdot g)$, when $X$ is a domain in a Riemannian manifold $(M,R\cdot g)$.  
\item[D)] Extend Theorem \ref{mainthm} to convex domains $X$ of low regularity.
\end{enumerate}

\noindent A non-technical presentation of the results in this article and open directions may be found in the blog posts \cite{blog1, blog2, blog3, blog4}.

\subsection*{Contents of the paper} 
The remainder of the paper is dedicated to proving Theorem \ref{mainthm}. The idea underlying its proof is to relate $\mathcal{M}_{X}(R)$ to the Dirichlet-to-Neumann operator of a boundary value problem first studied by Barcel\'o-Carbery \cite{barcarbs}. After a reformulation into a boundary value problem of order $n+1$ with parameter $R$, we adapt methods of semiclassical analysis and meromorphic Fredholm theory for explicit computations with a system of mixed higher-order operators.

The paper is organized as follows. We recall the Barcel\'o-Carbery boundary value problem in $\R^n \setminus X$ and its relation to the magnitude of $X$ in Subsection \ref{bcbvp}. To study the boundary value problem, we introduce boundary layer potentials in Subsection \ref{blaypot} -- these operators assemble into a solution operator by Proposition \ref{greensandal} and are used in Section \ref{reducttobound} to reduce the boundary value problem to the boundary $\partial X$.

The technical foundation for the proof of Theorem \ref{mainthm} is laid in Section \ref{reducttobound}. There, we show that the boundary layer potentials define a Calder\'on projector and a Dirichlet-to-Neumann operator. As these are matrix-valued pseudo-differential operators with parameter $R$, Theorem \ref{mainthm}c) on the existence of an asymptotic expansion of the magnitude function as $R\to +\infty$  immediately follows in Subsection \ref{asypexp}. The first three coefficients in part d) are computed in Subsection \ref{computingcoeffsubs} from the symbol of the Dirichlet-to-Neumann operator. By Equation \eqref{willertsasym} it suffices to show that $c_j(X)$ is  \emph{proportional} to the stated terms in part d) up to a prefactor which only depends on the dimension $n$. Part e) concerning the structure of the coefficients $c_j$, $j>2$, is proved in Subsection \ref{genstrusubs}. 

We finally prove part a) in Subsection \ref{meroextsubs} and discuss the role of scattering resonances mentioned in part b) in Subsection \ref{scattpolesubs}. The proofs rely on the explicit knowledge of the boundary layer potentials and their meromorphic dependence on $R\in \C$ in combination with the meromorphic Fredholm theorem.\\

\noindent \emph{Notation:} We write $f \lesssim g$ provided there exists a constant $C$ independent of $f$ and $g$ such that $f \leq Cg$. If $\mathcal{H}, \mathcal{H}_1, \mathcal{H}_2$ denote Hilbert spaces, the algebra of bounded operators from $\mathcal{H}_1$  to $\mathcal{H}_2$ is denoted by $\mathbb{B}(\mathcal{H}_1,\mathcal{H}_2)$, $\mathbb{B}(\mathcal{H}) = \mathbb{B}(\mathcal{H},\mathcal{H})$, and we write $\mathrm{id}_{\mathcal{H}}$ for the identity map on $\mathcal{H}$. We use $R$ to denote a parameter which is $>0$ or lies in a sector in $\mathbb{C}$, depending on context.

\large
\section{Magnitude and the Barcel\'o-Carbery boundary value problem}
\normalsize

In this section we study a boundary value problem introduced by Barcel\'o-Carbery in \cite{barcarbs}. Boundary layer potentials are used to reduce it to a problem on the boundary of $X$. Throughout this section we assume that $X\subseteq \R^n$ is a compact domain with $C^\infty$-boundary. We further assume that $n=2m-1$ is odd. We sometimes write $\Omega:=\R^n\setminus X$ and refer to $\Omega$ as \emph{the exterior domain}. Note that $\partial \Omega = \partial X$. We use the Sobolev spaces $H^s(\R^n):=(1-\Delta)^{-s/2}L^2(\R^n)$ of exponent $s\geq 0$. Here, the Laplacian $\Delta$ is given by $\Delta=\sum_{j=1}^n \frac{\partial^2}{\partial x_j^2}$. The spaces $H^s(X)$ and $H^s(\Omega)$ are defined using restrictions. The Sobolev spaces $H^s(\partial X)$ can be defined using local charts or as $(1-\Delta_{\partial X})^{-s/2}L^2(\partial X)$.

We use $\partial_\nu$ to denote the Neumann trace of a function $u$ in $\Omega$. By standard theory, $\partial_\nu$ extends to a continuous operator $H^s(\Omega)\to H^{s-3/2}(\partial X)$ for $s>3/2$. Similarly, we let $\gamma_0:H^s(\Omega)\to H^{s-1/2}(\partial X)$ denote the trace operator defined for $s>1/2$. 

\subsection{The Barcel\'o-Carbery boundary value problem}
\label{bcbvp}

In their work mentioned in the introduction, Barcel\'o-Carbery \cite{barcarbs} reduced the computation of the magnitude to a boundary value problem through the identity 
$$\textnormal{mag}(X)=\frac{\textnormal{vol}_n(X)}{n!\omega_n}+\frac{1}{n!\omega_n}\sum_{\frac{m}{2}<j\leq m} (-1)^j\begin{pmatrix} m\\ j\end{pmatrix} \int_{\partial X}\partial_\nu \Delta^{j-1} h\,\rd S .$$
Here $h\in H^m(\R^n)$ is the unique weak solution to 
\begin{equation}
\label{wekaformra}
(1-\Delta)^mh=0\quad\mbox{weakly in $\Omega=\R^n\setminus X$ and $h=1$ on $X$}.
\end{equation}
The integrals $\int_{\partial X}\partial_\nu \Delta^{j-1} h\,\rd S$ are defined from an exterior limiting procedure (see \cite[Section 5]{barcarbs}).

To better describe the asymptotic problem, we introduce the operators
$$\mathcal{D}^j_R:=\begin{cases}
\partial_\nu \circ(R^2-\Delta)^{(j-1)/2}, \; &\mbox{when $j$ is odd},\\
\gamma_0\circ(R^2-\Delta)^{j/2}, \; &\mbox{when $j$ is even}.
\end{cases}$$
By the trace theorem, the operators $\mathcal{D}^j_R$ are continuous as operators $\mathcal{D}^j_R:H^s(\Omega)\to H^{s-j-1/2}(\partial X)$ for $s>j+1/2$. Recall the notation $\Gamma_\alpha := \{ R \in \C\setminus \{0\}: |\mathrm{arg}(R)|<\alpha\}$.

\begin{prop}
\label{uniquesolu}
Let $R \in \Gamma_{\frac{\pi}{n+1}}$. Suppose that $u_j\in H^{2m-j-1/2}(\partial X)$, $j=0,\ldots, m-1$, are given. Then there is a unique weak solution $u\in H^{2m}(\Omega)$ to 
\begin{equation}
\label{thebvpgener}
\begin{cases}
(R^2-\Delta)^mu&=0\quad\mbox{in $\Omega$}\\
\mathcal{D}^j_R u&=u_j
\;\;, \; j=0,...,m-1.
\end{cases}
\end{equation}
\end{prop}

\begin{proof}
For real $R>0$ the proof is slightly more elementary and follows that of \cite[Proposition 1]{barcarbs}: Pick a $g\in H^m(X)$ such that $\mathcal{D}^j_R g=u_j$ for $j=0,...,m-1$. For $R>0$, we define the norm $\|v\|_{H^m_R(\R^n)}^2:=\int_{\R^n} (R^2+|\xi|^2)^m|\hat{v}(\xi)|^2\mathrm{d}\xi$. Consider the unique minimizer $u_0\in H^m(\R^n)$ to the extremal problem $\inf\{\|v\|_{H^m_R(\R^n)}^2: v|_X=g\}$. It follows by the Euler-Lagrange equations that $(R^2-\Delta)^mu_0=0$ in weak sense in $\Omega$. By elliptic regularity, $u:=u_0|_\Omega\in H^{2m}(\Omega)$. By construction, $\mathcal{D}^j_R u=u_j$, $j=0,...,m-1$. The weak solution $u$ is uniquely determined since $u_0$ is.

The general case $R \in \Gamma_{\frac{\pi}{n+1}}$ follows from the coercivity of the quadratic form $\int_{\R^n} (R^2+|\xi|^2)^m|\hat{v}(\xi)|^2\mathrm{d}\xi$ when $\mathrm{arg}(R^{2m})< \pi$. 
\end{proof}

We refer to \eqref{thebvpgener} as the Barcel\'o-Carbery boundary value problem. While \cite{barcarbs} stated the boundary value problem for bounded, convex $\Omega$, the convexity assumption may be omitted when $\partial \Omega$ is smooth.

\begin{prop}
\label{compformag}
Suppose that $h_R\in H^{2m}(\Omega)$ is the unique weak solution to the boundary value problem
\begin{align*}
\begin{cases}
(R^2-\Delta)^mh_R&=0\quad\mbox{in $\Omega$}\\
\vspace{-3mm}\\
\mathcal{D}^j_R h_R&=
\begin{cases} 
R^j, \;& j\mbox{   even}\\
0, \;& j\mbox{   odd}.
\end{cases}
\;\;, \; j=0,...,m-1.
\end{cases}
\end{align*}
Then the following identity holds 
$$\textnormal{mag}(R\cdot X)=\frac{\textnormal{vol}_n(X)}{n!\omega_n}R^n-\frac{1}{n!\omega_n}\sum_{\frac{m}{2}<j\leq m} R^{n-2j}\int_{\partial X}\mathcal{D}^{2j-1}_R h_R\,\rd S.$$
\end{prop}

\begin{proof}
The proof is a computational exercise starting from \cite[Theorem 5]{barcarbs}. As stated in \cite{barcarbs}, \cite[Theorem 5]{barcarbs} holds for compact convex domains but its proof readily extends to compact smooth domains. In particular, we have for any compact domain $Y$ with smooth boundary that 
\begin{equation}
\label{reqzero}
\textnormal{mag}(Y)=\frac{\textnormal{vol}_n(Y)}{n!\omega_n}+\frac{1}{n!\omega_n}\sum_{\frac{m}{2}<j\leq m} (-1)^j\begin{pmatrix} m\\j\end{pmatrix}\int_{\partial Y}\partial_\nu\Delta^{j-1} h_Y\,\rd S,
\end{equation}
where the boundary integrals are defined from an exterior limit and $h_Y\in H^m(\R^n)$ solves the problem \eqref{wekaformra}. By the same argument as in the proof of Proposition \ref{uniquesolu}, the problem \eqref{wekaformra} is equivalent to finding $h_Y|_{\R^n\setminus Y}\in H^{2m}(\R^n\setminus Y)$ which solves
\begin{align*}
\begin{cases}
(1-\Delta)^mh_Y&=0\quad\mbox{in $\R^n\setminus Y$}\\
\vspace{-3mm}\\
\mathcal{D}^j_0 h_Y&=
\begin{cases} 
1, \;& j=0\\
0, \;& j=1,...,m-1.
\end{cases}
\end{cases}
\end{align*}
Here we use the notation 
$$\mathcal{D}^j_0:=\begin{cases}
\partial_\nu \circ(-\Delta)^{(j-1)/2}, \; &\mbox{when $j$ is odd},\\
\gamma_0\circ(-\Delta)^{j/2}, \; &\mbox{when $j$ is even}.
\end{cases}$$

If we take $Y=R\cdot X$, a short computation with the change of variables $x\mapsto Rx$ shows that the function $h_R\in H^{2m}(\Omega)$ defined by $h_R(x):=h_{R\cdot X}(Rx)$ is the unique solution to the problem
\begin{align*}
\begin{cases}
(R^2-\Delta)^mh_R&=0\quad\mbox{in $\Omega$}\\
\vspace{-3mm}\\
\mathcal{D}^j_R h_R&=
\begin{cases} 
R^j, \;& j\mbox{   even}\\
0, \;& j\mbox{   odd}.
\end{cases}
\;\;, \; j=0,...,m-1.
\end{cases}
\end{align*}
Therefore, Equation \eqref{reqzero} implies the identity
\begin{align}
\label{reqzerowithd0}
\textnormal{mag}(R\cdot X)&=\frac{\textnormal{vol}_n(X)}{n!\omega_n}R^n+\frac{1}{n!\omega_n}\sum_{\frac{m}{2}<j\leq m} (-1)^j\begin{pmatrix} m\\j\end{pmatrix}\int_{R\cdot \partial X}\partial_\nu\Delta^{j-1} h_{R\cdot X}\,\rd S=\\
\nonumber
&=\frac{\textnormal{vol}_n(X)}{n!\omega_n}R^n-\frac{1}{n!\omega_n}\sum_{\frac{m}{2}<j\leq m} \begin{pmatrix} m\\j\end{pmatrix}R^{n-2j}\int_{ \partial X} \mathcal{D}^{2j-1}_0 h_{R}\,\rd S
\end{align}
What needs to be proven is that the relation \eqref{reqzerowithd0}, formulated in terms of $\mathcal{D}^j_0$, can be rewritten as stated in the proposition using the boundary conditions $\mathcal{D}^j_R$.

The binomial theorem gives us $\mathcal{D}^{2j-1}_R=\sum_{l=1}^{j} \begin{pmatrix} j-1\\ l \end{pmatrix} R^{2j-2l}\mathcal{D}^{2l-1}_0$. Noting that $\mathcal{D}^{2l-1}_0 h_R=0$ for $l\leq m/2$ we arrive at the identity
\begin{align*}
\sum_{\frac{m}{2}<j\leq m} R^{n-2j}\int_{\partial X}\mathcal{D}^{2j-1}_R h_R\,\rd S&=\sum_{\frac{m}{2}<j\leq m}\sum_{l=1}^{j} \begin{pmatrix} j-1\\ l-1 \end{pmatrix} R^{n-2l}\int_{\partial X}\mathcal{D}^{2l-1}_0 h_R\,\rd S=\\
&=\sum_{\substack{\frac{m}{2}<j\leq m,\\ \frac{m}{2}<l\leq j}} \begin{pmatrix} j-1\\ l-1 \end{pmatrix} R^{n-2l}\int_{\partial X}\mathcal{D}^{2l-1}_0 h_R\,\rd S=\\
&=\sum_{\frac{m}{2}<l\leq m}\sum_{j=l}^{m} \begin{pmatrix} j-1\\ l-1 \end{pmatrix} R^{n-2l}\int_{\partial X}\mathcal{D}^{2l-1}_0 h_R\,\rd S.
\end{align*}
By the hockey stick identity, 
$$\sum_{j=l}^{m} \begin{pmatrix} j-1\\ l-1 \end{pmatrix} =\begin{pmatrix} m\\ l \end{pmatrix}.$$
Therefore, 
\begin{align*}
\sum_{\frac{m}{2}<j\leq m} R^{n-2j}\int_{\partial X}\mathcal{D}^{2j-1}_R h_R\,\rd S&=\sum_{\frac{m}{2}<l\leq m} \begin{pmatrix} m\\ l \end{pmatrix} R^{n-2l}\int_{\partial X}\mathcal{D}^{2l-1}_0 h_R\,\rd S,
\end{align*}
and the proposition follows.
\end{proof}

Our analysis of the magnitude function $\mathcal{M}_X(R):= \textnormal{mag}(R\cdot X)$ is based on a precise description of the solution $h_R$ as $R$ varies.

\subsection{The layer potentials}
\label{blaypot}

The main technical tool we will use to move from a problem on $\Omega$ to $\partial X$ is that of boundary layer potentials. We consider the function 
$$K(R;z):=\frac{\kappa_n}{R} \mathrm{e}^{-R|z|}, \quad z\in \R^n.$$
Here $\kappa_n>0$ is a normalizing constant chosen such that
$$(R^2-\Delta)^mK=\delta_0$$  in the sense of distributions on $\R^n$. The function $K(R;\cdot)$ is a fundamental solution to $(R^2-\Delta)^m$. By considerations in Fourier space, one readily deduces that the convolution operator $f\mapsto K(R;\cdot)*f$ defines a holomorphic function $\C_+=\{R: \mathrm{Re}(R)>0\} \to \mathbbm{B}(L^2(\R^n))$ which coincides with $(R^2-\Delta)^{-m}$. 

For $l=0,\ldots,n$, we define the functions
$$K_l(R;x,y):=(-1)^l\mathcal{D}^{n-l}_{R,y}K(R;x-y), \quad x\in \R^n, \; y\in \partial X.$$
Here $\mathcal{D}^l_{R,y}$ denotes $\mathcal{D}^l_R$ acting in the $y$-variable. We also consider the distributions
$$K_{j,k}(R;x,y):= \mathcal{D}^j_{R,x}K_{k}(R;x,y), \quad x\in \partial X.$$

\begin{deef}
For $f\in C^\infty(\partial X)$ and $R\in \C\setminus\{0\}$, we define the operators 
\begin{align*}
\mathcal{A}_k(R)f(x)&:=\int_{\partial X} K_k(R;x,y)f(y)\mathrm{d}S(y), \quad x\in \Omega ,\\
A_{j,k}(R)f(x)&:=\int_{\partial X} K_{j,k}(R;x,y)f(y)\mathrm{d}S(y), \quad x\in \partial X .
\end{align*}
The integral defining $A_{j,k}(R)$ is defined in the sense of an exterior limit.
\end{deef}

We remark that $A_{j,k}(R):C^\infty(\partial X)\to C^\infty(\partial X)$ and $\mathcal{A}_k(R):C^\infty(\partial X)\to C^\infty(\Omega)$ for any $R\in \C\setminus\{0\}$.

\begin{lem}
\label{exteningtoomega} 
For any $f\in C^\infty(\partial X)$, $j= 0, \ldots, n$ and $R\in \C\setminus\{0\}$, the function $u:=\mathcal{A}_{j}(R)f\in C^\infty(\Omega)$ satisfies
$$(R^2-\Delta)^mu=0 \quad\mbox{in $\Omega$}.$$
For $R \in \C_+$, the operator $\mathcal{A}_{j}(R)$ extends to a continuous operator 
$$\mathcal{A}_{j}(R):H^s(\partial X)\to H^{s+j+1/2}(\Omega), \quad s>0.$$
\end{lem}

\begin{proof}
It is clear from the construction that $u:=\mathcal{A}_{j}(R)f$ satisfies $(R^2-\Delta)^mu=0$ in $\Omega$. By elliptic regularity, $u$ is a smooth function in $\Omega$. Decompose $\overline{\Omega}=\overline{\Omega'}{\cup}\overline{\Omega''}$ where $\Omega'$ is a bounded domain with smooth boundary and $\Omega''$ is a smooth domain with boundary, compact complement, $\mathrm{d}(X,\Omega'')=r>0$. 
We have that 
$$\|u\|_{H^{s+j+1/2}(\Omega)}^2\lesssim\|u\|_{H^{s+j+1/2}(\Omega')}^2+\|u\|_{H^{s+j+1/2}(\Omega'')}^2.$$
By elliptic regularity on smooth, compact domains, $\|u\|_{H^{s+j+1/2}(\Omega')}\lesssim \|f\|_{H^s(\partial X)}$. Since $K_{j}(R;x,y)$ is smooth when $|x-y|>r$ and the kernel decays exponentially in the off-diagonal direction for $R \in \C_+$, one easily deduces $\|u\|_{H^{s+j+1/2}(\Omega'')}\lesssim \|f\|_{H^s(\partial X)}$ (as in, for instance, \cite[Lemma 1.6]{goffsch}).
\end{proof}

\begin{prop}
\label{greensandal}
Let $R \in \C_+$. If $u\in H^{2m}(\Omega)$ solves $(R^2-\Delta)^mu=0$ in $\Omega$, then 
$$u=\sum_{l=0}^n \mathcal{A}_l(R)\left(\mathcal{D}_R^lu\right).$$
\end{prop}

\begin{proof}
We prove the proposition using Green's formula. Let $(\cdot,\cdot)_\Omega$ denote the bilinear pairing of distributions with test functions on $\Omega$. For $\varphi\in C^\infty_c(\overline{\Omega})$ and $x\in \Omega$, we apply Green's formula to write
\begin{align*}
\varphi(x)&=(\delta_x,\varphi)_\Omega=((R^2-\Delta_y)^mK(R;\cdot-x),\varphi)_\Omega=\\
&=((R^2-\Delta_y)^{m-1}K(R;\cdot-x),(R^2-\Delta)\varphi)_\Omega+\mathcal{A}_0(R)\left(\mathcal{D}_R^0\varphi\right)(x)+\mathcal{A}_1(R)\left(\mathcal{D}_R^1\varphi\right)(x).
\end{align*}
Proceeding by induction, we obtain 
\begin{align*}
\varphi(x)=(K(R;\cdot-x),(R^2-\Delta)^m\varphi)_\Omega+\sum_{l=0}^n \mathcal{A}_l(R)\left(\mathcal{D}_R^l\varphi\right)(x).
\end{align*}
By Lemma \ref{exteningtoomega} and a density argument, we conclude for $u\in H^{2m}(\Omega)$ and $\mathrm{Re}(R)>0$ that
\begin{align*}
u(x)=(K(R;\cdot-x),(R^2-\Delta)^mu)_\Omega+\sum_{l=0}^n \mathcal{A}_l(R)\left(\mathcal{D}_R^lu\right)(x).
\end{align*}
In particular, if $u\in H^{2m}(\Omega)$ solves $(R^2-\Delta)^mu=0$ in $\Omega$, then
\begin{align*}
u(x)=\sum_{l=0}^n \mathcal{A}_l(R)\left(\mathcal{D}_R^lu\right)(x).
\end{align*}
\end{proof}

\large
\section{Reduction to the boundary}
\normalsize
\label{reducttobound}

The Barcel\'o-Carbery boundary value problem \eqref{thebvpgener} can be reduced to an elliptic pseudo-differential problem on the boundary. We work under the same assumptions as in the previous section: $X\subseteq \R^n$ is a smooth, compact domain, $n=2m-1$ and $\Omega:=\R^n\setminus X$. On the closed manifold $\partial\Omega = \partial X$, the problem is reduced to a study of pseudo-differential operators with parameters, a natural tool for explicit computations of asymptotic expansions. Pseudo-differential methods are standard in the literature, and the reader can find the basic ideas in the text book \cite[Part II and III]{GrubbDistOp}. We recall the salient properties of pseudo-differential operators with parameters in Appendix \ref{briefintropara}.

\subsection{The boundary operators}

We consider a sector $\Gamma \subset \C_+$. The operators $A_{jk}$ will turn out to be pseudo-differential operators with parameter $R\in \Gamma$. The theory of pseudo-differential operators with parameter is well developed and treats the parameter $R$ as an additional covariable. Asymptotic expansions of symbols are done in $(R,\xi)$ simultaneously. References can be found in Appendix \ref{briefintropara}.

\begin{thm}
\label{theboundaryops}
For any $R\in \C \setminus \{0\}$, the operators $A_{j,l}(R)$ are pseudo-differential operators of order $j-l$ on $\partial X$. In fact, $A_{j,l}\in \Psi^{j-l}(\partial X;\Gamma)$ are pseudo-differential operators with parameter $R\in \Gamma$ of order $j-l$. Moreover, the following symbol computations hold:
\begin{enumerate}
\item[a.)] The {\bf principal symbol} of $A_{j,l}$ as a pseudo-differential operator with parameter satisfies 
$$\sigma_{j-l,R}(A_{j,l})(x',\xi',R)=c_{j,l}(R^2+|\xi'|^2)^{(j-l)/2}, \; (x',\xi')\in T^*\partial X,$$
where $(c_{j,l})_{j,l=0}^n\in M_{2m}(\C)$ is a constant matrix independent of $X$, whose entries only depend on $n$.
\item[b.)] In suitable coordinates, the part of the {\bf full symbol of order $j-l-1$} of $A_{j,l}$ as a pseudo-differential operator with parameter satisfies 
$$\sigma_{j-l-1,R}(A_{j,l})(x',\xi',R)=c_{j,l,-1}H(x')(R^2+|\xi'|^2)^{(j-l-1)/2}, \; (x',\xi')\in T^*\partial X,$$
where $(c_{j,l,-1})_{j,l=0}^n\in M_{2m}(\C)$ is a constant matrix independent of $X$, whose entries only depend on $n$, and $H$ denotes the mean curvature of $\partial X$.
\end{enumerate}
\end{thm}

In the statement of Theorem  \ref{theboundaryops} we use the notation $(x',\xi')$ for coordinates on the cotangent bundle $T^*\partial X$ to distinguish from the Cartesian coordinates on the ambient Euclidean space $\R^n$. The proof of Theorem \ref{theboundaryops} is of a computational nature and may be found in Subsection \ref{theorem9} of the appendix. We define the Hilbert space 
$$\mathcal{H}:=\underbrace{\bigoplus_{j=0}^{m-1}H^{2m-j-1/2}(\partial X)}_{\mathcal{H}_+}\oplus \underbrace{\bigoplus_{j=m}^{n}H^{2m-j-1/2}(\partial X)}_{\mathcal{H}_-}.$$
 By Theorem \ref{theboundaryops} the $2m\times 2m$-matrix of operators $\mathbb{A}:=(A_{jl})_{j,l=0}^n$ can be considered as an operator $\mathbb{A}:\mathcal{H}\to \mathcal{H}$. This operator decomposes into a matrix of operators
$$\mathbb{A}=\begin{pmatrix} \mathbb{A}_{++}& \mathbb{A}_{+-}\\ \mathbb{A}_{-+}& \mathbb{A}_{--}\end{pmatrix}:\begin{matrix} \mathcal{H}_+\\\oplus\\\mathcal{H}_-\end{matrix}\longrightarrow \begin{matrix} \mathcal{H}_+\\\oplus\\\mathcal{H}_-\end{matrix}.$$
Here $\mathbb{A}_{pq}:\mathcal{H}_q\to \mathcal{H}_p$ for $p,q\in \{+,-\}$.

We note here that the entries $A_{jl}$ of $\mathbb{A}$ are of order $j-l$, and hence depend on their position in the matrix. The appropriate notion of an elliptic operator (with parameter) for such mixed-order  systems goes back to Agmon-Douglis-Nirenberg \cite{agmon, grubb77}, see also \cite[Chapter XIX.5]{hormanderIII}. The operator $\mathbb{A}$ is called the Calder\'on projector of the Barcel\'o-Carbery boundary value problem \eqref{thebvpgener}. It allows to reduce problem \eqref{thebvpgener} to the boundary $\partial X$. In Lemma \ref{caldprojisproj} we show that $\mathbb{A}$ is a projection.

\begin{deef}
For $R\in \Gamma$, define the operators 
\begin{align*}
T:\mathcal{H}&\to H^{2m}(\Omega), \quad v=(v_j)_{j=0}^n\mapsto \sum_{l=0}^n \mathcal{A}_l(v_l),\\
r: H^{2m}(\Omega)&\to \mathcal{H}, \quad u\mapsto (\mathcal{D}^j_Ru)_{j=0}^n.
\end{align*}
We define the space of Cauchy data of the Barcel\'o-Carbery boundary value problem \eqref{thebvpgener} as $\mathcal{C}_{BC}(R):=r\ker(R^2-\Delta)^m\subseteq \mathcal{H}$.
\end{deef}

The operator $T$ is well defined by Lemma \ref{exteningtoomega}, and $r$ is well defined by the trace theorem.

\begin{lem}
\label{caldprojisproj}
For any $R\in \Gamma$, the operator $\mathbb{A}$ is a projection onto the space of Cauchy data $\mathcal{C}_{BC}(R)$. In fact, $\mathbb{A}=r\circ T$, and $T$ is a left inverse of $r$ on the closed subspace $\ker(R^2-\Delta)^m\subseteq H^{2m}(\Omega)$. 
\end{lem}

\begin{proof}
Since the image of $T$ satisfies $\im(T)\subseteq \ker(R^2-\Delta)^m\subseteq H^{2m}(\Omega)$, $\mathbb{A}$ is a projection if $\mathbb{A}=r\circ T$ and $T$ is the left inverse of $r$ on the closed subspace $\ker(R^2-\Delta)^m\subseteq H^{2m}(\Omega)$. The identity $T\circ r|_{\ker(R^2-\Delta)^m}=\mathrm{id}_{\ker(R^2-\Delta)^m}$ follows from Proposition \ref{greensandal}.
\end{proof}

\begin{lem}
\label{calderprojlemacomp}
For any $R\in \C\setminus\{0\}$, the entries in the matrix $\mathbb{A}$ of operators, $\mathbb{A}_{++}$, $ \mathbb{A}_{+-}$, $\mathbb{A}_{-+}$ and $\mathbb{A}_{--}$, are elliptic pseudo-differential operators. Moreover when restricting the parameter to the range $R\in \Gamma$, the operators $\mathbb{A}_{++}$, $ \mathbb{A}_{+-}$, $\mathbb{A}_{-+}$ and $\mathbb{A}_{--}$ define elliptic pseudo-differential operators with parameter $R\in \Gamma$.
\end{lem}

\begin{proof}
The symbol computation of Theorem \ref{theboundaryops} shows that $\mathbb{A}_{++}$, $ \mathbb{A}_{+-}$, $\mathbb{A}_{-+}$ and $\mathbb{A}_{--}$, are elliptic pseudo-differential operators for some $R$, if and only if they are elliptic for all $R$, if and only if the operators $\mathbb{A}_{++}$, $ \mathbb{A}_{+-}$, $\mathbb{A}_{-+}$ and $\mathbb{A}_{--}$ define elliptic pseudo-differential operators with parameter $R\in \Gamma$. In fact, the principal symbols of $\mathbb{A}_{++}$, $ \mathbb{A}_{+-}$, $\mathbb{A}_{-+}$ and $\mathbb{A}_{--}$ only depend on the geometry of $\partial X$ and its embedding into a compact neighborhood in $\R^n$. As such, the lemma is reduced to the situation of \cite{Grubb71}. By \cite[Section 6]{Grubb71} (see also \cite[Chapter 11]{GrubbDistOp}), the operators $\mathbb{A}_{++}$, $ \mathbb{A}_{+-}$, $\mathbb{A}_{-+}$ and $\mathbb{A}_{--}$, are elliptic pseudo-differential operators as they arise from a boundary value problem which is elliptic in the Lopatinskii-Shapiro sense.
\end{proof}

\begin{cor}
\label{invertingaplumi}
{ For any sector $\Gamma\subseteq \C_+$, there is a $C_\Gamma\geq 0$ such that when $R\in \Gamma$ satisfies $\mathrm{Re}(R)\geq C_\Gamma$}, the operators $\mathbb{A}_{++}$, $ \mathbb{A}_{+-}$, $\mathbb{A}_{-+}$ and $\mathbb{A}_{--}$ are invertible.
\end{cor}

\begin{proof}
If $P$ is an elliptic pseudo-differential operator with parameter $R\in \Gamma$, it admits a parametrix $Q$, which is an elliptic pseudo-differential operator with parameter $R\in \Gamma$, such that $1-PQ$ and $1-QP$ are smoothing pseudo-differential operators with parameter $R\in \Gamma$. By taking $\mathrm{Re}(R)$ large enough, the smoothing operators $1-PQ$ and $1-QP$ can be made arbitrarily small in operator norm since their kernels are in $\Psi^{-\infty}(M;\Gamma)=\mathcal{S}(\Gamma;C^\infty(M\times M))$. A Neumann series argument shows that $P$ has a left and a right inverse. Now the corollary follows readily from Lemma \ref{calderprojlemacomp} by setting $P=\mathbb{A}_{++}$, $ \mathbb{A}_{+-}$, $\mathbb{A}_{-+}$ and $\mathbb{A}_{--}$.
\end{proof}

\subsection{The magnitude and the Dirichlet-to-Neumann operator}

When computing the magnitude of a domain in $\R^n$, by Proposition \ref{compformag} we need to compute the ``Neumann part" of the boundary data $(\mathcal{D}^j_Rh_R)_{j=m}^n\in \mathcal{H}_-$ from the given ``Dirichlet part" of the boundary data $(\mathcal{D}^j_Rh_R)_{j=0}^{m-1}\in \mathcal{H}_+$. It follows from Proposition \ref{uniquesolu} that $(\mathcal{D}^j_Rh_R)_{j=m}^n\in \mathcal{H}_-$ is uniquely determined by $(\mathcal{D}^j_Rh_R)_{j=0}^{m-1}\in \mathcal{H}_+$ for $R\in \Gamma_{\frac{\pi}{n+1}}$. This map is implemented by the standard construction of a Dirichlet-to-Neumann operator, and we later continue this operator meromorphically to $R\in \C$. 

\begin{deef} 
Define $\Lambda(R):\mathcal{H}_+\to \mathcal{H}_-$ by setting 
$$\Lambda(R)(u_j)_{j=0}^{m-1}:=(\mathcal{D}^j_Ru)_{j=m}^n,$$
where $u\in H^{2m}(\Omega)$ solves Equation \eqref{thebvpgener}.
\end{deef}

\begin{thm}
\label{lambdaisps}
Let $\Gamma\subseteq \Gamma_{\frac{\pi}{n+1}}$ be a sector. The operator $\Lambda=(\Lambda_{j+m,l})_{j,l=0}^{m-1}$ is a pseudo-differential operator with parameter $R\in \Gamma$. In fact, $\Lambda_{j+m,l}\in \Psi^{j+m-l}(\partial X;\Gamma)$. Moreover, for $\mathrm{Re}(R)>0$ large enough, $\Lambda(R)=\mathbb{A}_{+-}(R)^{-1}(1-\mathbb{A}_{++}(R))$.
\end{thm}

\begin{proof}
By construction, the space of Cauchy data $\mathcal{C}_{BC}(R)$ is the graph of $\Lambda(R)$. By Lemma \ref{caldprojisproj}, $\mathbb{A}$ projects onto the space of Cauchy data $\mathcal{C}_{BC}(R)$. This implies the identity $\mathbb{A}_{++}(R)+\mathbb{A}_{+-}(R)\Lambda(R)=1$. The identity $\Lambda(R)=\mathbb{A}_{+-}(R)^{-1}(1-\mathbb{A}_{++}(R))$ for large $\mathrm{Re}(R)$ follows from Corollary \ref{invertingaplumi}. The Agmon-Douglis-Nirenberg \cite{agmon, grubb77} calculus for mixed order systems and ellipticity of $\mathbb{A}_{+-}$ imply that there is a matrix of parameter dependent pseudo-differential operators $\mathbb{B}(R)=(B_{j+m,k})_{j,k=0}^{m-1}$, where $B_{j+m,k}\in \Psi^{j+m-k}(\partial X;\Gamma)$, such that $\mathbb{B}(R)=\mathbb{A}_{+-}(R)^{-1}$ for large $\mathrm{Re}(R)$ and $1-\mathbb{B}(R)\mathbb{A}_{+-}(R), 1-\mathbb{A}_{+-}(R)\mathbb{B}(R) \in \Psi^{-\infty}(\partial X;\Gamma)$. As $\mathbb{A}_{+-}(R)\Lambda(R)=1-\mathbb{A}_{++}(R)$ we have that 
$$\Lambda_{j+m,l}-\sum_{k=0}^{m-1} B_{j+m,k}(\delta_{k,l}-A_{k,l})\in \Psi^{-\infty}(\partial X;\Gamma).$$
In fact, $\Lambda_{j+m,l}=\sum_{k=0}^{m-1} B_{j+m,k}(\delta_{k,l}-A_{k,l})$ for large $\mathrm{Re}(R)$. The property $\Lambda_{j+m,l}\in \Psi^{j+m-l}(\partial X;\Gamma)$ can be deduced from the fact that $B_{j+m,k}\in \Psi^{j+m-k}(\partial X;\Gamma)$ and $\delta_{k,l}-A_{k,l}\in \Psi^{k-l}(\partial X,\Gamma)$.

\end{proof}

\large
\section{Asymptotic expansion of $\mathcal{M}_X$}
\normalsize

The results of the previous section relate the magnitude function to integrals involving the parameter dependent Dirichlet-to-Neumann operator via Proposition \ref{compformag} (see page \pageref{compformag}). Using this fact, we now prove that $\mathcal{M}_X$ admits an asymptotic expansion. The key ingredient is Lemma \ref{condexplem} which shows that integrals involving a parameter dependent pseudo-differential operator admit asymptotic expansions in the semiclassical limit $R\to +\infty$.

\subsection{Existence of the asymptotic expansion}
\label{asypexp}

From the construction of the Dirichlet-to-Neumann operator, we can asymptotically compute the magnitude using Proposition \ref{compformag}. The latter proposition implies that 
\begin{equation}
\label{expandinmagindn}
\mathcal{M}_X(R)=\frac{\textnormal{vol}_n(X)}{n!\omega_n}R^n-\frac{1}{n!\omega_n}\sum_{\frac{m}{2}<j\leq m}\sum_{0\leq l<m/2} R^{n-2j+2l}\int_{\partial X}\Lambda_{2j-1,2l}(R)(1)\,\rd S.
\end{equation}
Here $1$ denotes the constant function on $\partial X$. In the next lemma we compute the integrals $\int_{\partial X}\Lambda_{2j-1,2l}(R)(1)\,\rd S$, using the calculus with parameter $R$. If $f=f(R)$ is a function such that $f(R)=\mathcal{O}(R^{-M})$ for any $M\in \N$, we write $f(R)=\mathcal{O}(R^{-\infty})$.

\begin{lem}
\label{condexplem}
Let $\Gamma$ be a sector in $\C_+$ containing $\R_+$. Suppose that $a=a(x,\xi,R)\in C^\infty(T^*M\times \Gamma)$ is a full symbol with parameter $R\in \Gamma$ of a scalar pseudo-differential operator $A$ with parameter $R\in \Gamma$ on the closed manifold $M$. Then it holds that 
$$\langle A1,1\rangle_{L^2(M)}-\int_M a(x,0,R)\mathrm{d}x=\mathcal{O}(R^{-\infty}),$$
as $\mathrm{Re}(R)\to +\infty$ in $\Gamma$. In particular, if $a\sim \sum_{j\in \N} \sigma_{m-j}(a)$ is the asymptotic expansion of $a$ into homogeneous symbols $\sigma_{m-j}(a)=\sigma_{m-j}(a)(x,\xi,R)=R^{m-j}\sigma_{m-j}(a)(x,\xi/R,1)$ of order $m-j$, then 
$$\langle A1,1\rangle_{L^2(M)}\ \sim \ \sum_{j=0}^\infty \alpha_j R^{m-j},\quad\mbox{where}\quad \alpha_j:=\int_M \sigma_{m-j}(a)(x,0,1)\rd x.$$
\end{lem}

For the proof of Lemma \ref{condexplem} we refer to Subsection \ref{lemma16} of the appendix. The following proposition shows that $\mathcal{M}_X(R)$ admits an asymptotic expansion as $\mathrm{Re}(R)\to +\infty$ in a sector $\Gamma\subseteq \C_+$. We omit the proof, as it follows immediately from Theorem \ref{lambdaisps}, Equation \eqref{expandinmagindn} and Lemma \ref{condexplem}.

\begin{prop}
\label{thecoeffi}
Let $X\subseteq \R^n$ be a smooth, compact domain, $n=2m-1$, and  $\Gamma$ a sector in $\C_+$ containing $\R_+$. Then as $\mathrm{Re}(R)\to +\infty$ in $\Gamma$, 
$$\mathcal{M}_X(R)\ \sim \ \frac{\textnormal{vol}_n(X)}{n!\omega_n}R^n+\sum_{k=1}^\infty c_k R^{n-k},$$
where 
$$c_k:=\frac{1}{n!\omega_n}\sum_{\frac{m}{2}<j\leq m}\sum_{0\leq l<m/2}\int_{\partial X}\sigma_{2j-2l-k}(\Lambda_{2j-1,2l})(x,0,1)\,\rd S.$$
\end{prop}

\subsection{Computing the coefficients $c_1$ and $c_2$}
\label{computingcoeffsubs}

Given Willerton's expansion $$n! \mathcal{M}_{B_n}(R) = R^n + \frac{n (n+1)}{2} R^{n-1} + \frac{n (n+1)^2 (n-1)}{8} R^{n-2} + \mathcal{O}(R^{n-3})$$ for  the unit ball in odd dimension $n$ \cite{will}, it suffices to determine $c_1$ and $c_2$ up to prefactors which only depend on the dimension.

\begin{prop}
There is a universal number $\beta_n$ only depending on the dimension $n$ such that 
$$c_1=\beta_n\mathrm{vol}_{n-1}(\partial X).$$
\end{prop}

\begin{proof}
By Proposition \ref{thecoeffi}, $c_1=\frac{1}{n!\omega_n}\sum_{\frac{m}{2}<j\leq m}\sum_{0\leq l<m/2}\int_{\partial X}\sigma_{2j-2l-1}(\Lambda_{2j-1,2l})(x,0,1)\,\rd S$. However, $\sigma_{2j-2l-1}(\Lambda_{2j-1,2l})$ is the top degree symbol and as such it is computed from the top degree symbols of $\mathbbm{A}_{+-}$ and $\mathbbm{A}_{++}$. By Theorem \ref{theboundaryops} these are constants independent of $X$. As such, for some universal constants $\beta_{j,l}$ (depending only on $n$), $\sigma_{2j-2l-1}(\Lambda_{2j-1,2l})=\beta_{j,l}$. The theorem follows upon setting 
$$\beta_n:=\frac{1}{n!\omega_n}\sum_{\frac{m}{2}<j\leq m}\sum_{0\leq l<m/2}\beta_{j,l}.$$
\end{proof}

The next result follows from a computation with products of pseudo-differential operators, see \cite[Theorem 18.1.9]{hormanderIII}.

\begin{lem}
\label{symbolcomputation}
Let $M$ be a closed manifold and $\Gamma\subseteq \C$ a sector. Suppose that $A_+\in \Psi^{m_+}(M;\Gamma)$ is invertible and that $A_-\in \Psi^{m_-}(M;\Gamma)$. Let $a_\pm\in C^\infty(T^*M\times \Gamma)$ denote full symbols of $A_\pm$ with asymptotic expansions $a_\pm\sim \sum_{j\geq 0} a_{\pm, m_\pm-j}$ with $a_{\pm, m_\pm}$ only depending on $R^2+|\xi|^2$ (for some metric on $M$).
\begin{enumerate}
\item[a)] A symbol $q = \sum_{j\in \N_0} q_{-m_+ - j}$ of $A_+^{-1}$ is determined recursively by
\begin{equation*}
q_{-m_+} = a_{+,m_+}^{-1}
\end{equation*}
\begin{equation} \label{recursionrelation}
q_{-m_+-j} = -q_{-m_+}\sum_{\substack{k+l+t=j\\l<j}}
\sum_{|\alpha| =t} \frac{(-i)^{|\alpha|}}{\alpha !} \, \partial_\xi^\alpha
a_{+, m_+-k} \, \partial_x^\alpha q_{-m_+-l}.
\end{equation}
\item[b)] A full symbol $p \ \sim \ \sum_{j\in \N_0} p_{-m_++m_- - j}$ of $A_+^{-1}A_-$ is given by 
\begin{equation} \label{fullsymb}
p_{-m_++m_- - j} = \sum_{k+l+t=j } \sum_{|\alpha|=t} \frac{(-i)^{|\alpha|}}{\alpha!}\partial_\xi^\alpha q_{-m_+-k} \partial_x^\alpha a_{-, m_- - l}.
\end{equation}
\item[c)] In particular, for a symbol $b$ of order $m_--m_+-2$, the expression 
$$(a_{+,m_+}^{-1}-a_{+,m_+}^{-1}a_{+,m_+-1}a_{+,m_+}^{-1})a_{-,m_-}-a_{+,m_+}^{-1}a_{-,m_--1}+b,$$
is a full symbol of $A_+^{-1}A_-$.
\end{enumerate}
\end{lem}

\begin{prop}
There is a universal number $\gamma_n$ only depending on the dimension $n$ such that 
$$c_2=\gamma_n\int_{\partial X} H \mathrm{d}S,$$
where $H$ is the mean curvature of $\partial X$.
\end{prop}

\begin{proof}
By Proposition \ref{thecoeffi}, 
$$c_2=\frac{1}{n!\omega_n}\sum_{\frac{m}{2}<j\leq m}\sum_{0\leq l<m/2}\int_{\partial X}\sigma_{2j-2l-2}(\Lambda_{2j-1,2l})(x,0,1)\,\rd S.$$ 
We first analyze the symbol $\sigma_{2j-2l-2}(\Lambda_{2j-1,2l})$. Following the notation in the proof of Theorem \ref{lambdaisps}, we write $\mathbb{A}_{+-}(R)^{-1}=(B_{j+m,k})_{j,k=0}^{m-1}$ for large $\mathrm{Re}(R)$, where $B_{j+m,k}\in \Psi^{j+m-k}(\partial X;\Gamma)$. Lemma \ref{symbolcomputation} implies that for $\frac{m}{2}<j\leq m$, $0\leq l<m/2$,
\small
\begin{align}
\nonumber
\sigma_{2j-2l-2}(\Lambda_{2j-1,2l})=&\sum_{k_1,k_2,k_3=0}^{m-1}\sigma_{2j+m-1-k_1}(B_{2j+m-1,k_1})\sigma_{k_1-m-k_2-1}(A_{k_1,m+k_2})\sigma_{k_2+m-k_3}(B_{k_2+m,k_3})\sigma_{k_3-2l}(\delta_{k_3,2l}-A_{k_3,2l})-\\
\label{lowsymcomput}
&-\sum_{k=0}^{m-1}\sigma_{2j+m-1-k}(B_{2j+m-1,k})\sigma_{k-2l-1}(\delta_{k,2l}-A_{k,2l})=\\
\nonumber
=&\sum_{k_1,k_2,k_3=0}^{m-1}\sigma_{2j+m-1-k_1}(B_{2j+m-1,k_1})\sigma_{k_1-m-k_2-1}(A_{k_1,m+k_2})\sigma_{k_2+m-k_3}(B_{k_2+m,k_3})\sigma_{k_3-2l}(\delta_{k_3,2l}-A_{k_3,2l})+\\
\nonumber
&+\sum_{k=0}^{m-1}\sigma_{2j+m-1-k}(B_{2j+m-1,k})\sigma_{k-2l-1}(A_{k,2l}).
\end{align}
\normalsize
In the last identity we used that $\sigma_{k-2l-1}(\delta_{k,2l})=0$ which follows from the facts that $\delta_{k,2l}=0$ if $k\neq 2l$ and, as the lower order symbols of the identity operator vanishes, $\sigma_{k-2l-1}(\delta_{k,2l})=\sigma_{k-2l-1}(\mathrm{id})=0$ when $k=2l$. 

Let us turn to the symbols appearing in Equation \eqref{lowsymcomput}. By Theorem \ref{theboundaryops}, the principal symbols $(\sigma_{j+m-k}(B_{j+m,k}))_{j,k=0}^{m-1}$ and $(\sigma_{j-k}(A_{j,k}))_{j,k=0}^{m-1}$ depend on the dimension $n$ but are independent of $X$. More precisely, the $(j,k)$-th entry of $(\sigma_{j+m-k}(B_{j+m,k}))_{j,k=0}^{m-1}$ is a constant $X$-independent multiple of $(R^2-|\xi'|^{2})^{(j+m-k)/2}$ and the $(j,k)$-th entry of $(\sigma_{j-k}(A_{j,k}))_{j,k=0}^{m-1}$ is a constant $X$-independent multiple of $(R^2-|\xi'|^{2})^{(j-k)/2}$. Moreover, Theorem \ref{theboundaryops} shows that the $(j,k)$-th entry of $(\sigma_{j-k-1}(A_{j,k}))_{j,k=0}^{m-1}$ is a constant $X$-independent multiple of $H(x')(R^2-|\xi'|^{2})^{(j-k-1)/2}$. 

Combining our discussion in the previous paragraph with Equation \eqref{lowsymcomput}, we conclude that there are constants $\gamma_{j,l}$ independent of $X$ (depending only on $n$), such that 
$$\sigma_{2j-2l-2}(\Lambda_{2j-1,2l})(x',\xi',R)=\gamma_{j,l}H(x')(R^2+|\xi'|^2)^{j-l-1}.$$ 
The theorem follows upon setting 
$$\gamma_n:=\frac{1}{n!\omega_n}\sum_{\frac{m}{2}<j\leq m}\sum_{0\leq l<m/2}\gamma_{j,l}.$$
\end{proof}

From the explicit formulas for $c_j$, when $X$ is convex we conclude that $c_j$ is proportional to the intrinsic volume $V_{n-j}(X)$ for $j=0,1,2$ \cite[pp.~210]{schneider}.

\subsection{General structure of $c_j$}
\label{genstrusubs}

The formula \eqref{fullsymb}, combined with \eqref{recursionrelation} and \eqref{lowerorder} (see below on page \pageref{lowerorder}), allows to deduce the general structure of the coefficients $c_j$.

To do so, we pick local coordinates at a point on $\partial X$. We can assume that this point is $0\in \R^n$ and that the coordinates are of the form $(x',S(x'))$, where $x'$ belongs to some neighborhood of $0\in \R^{n-1}$ and $S$ is a scalar function with $S(0)=0$ and $\grad S(0)=0$. We say that product of derivatives of the metric $|\xi'|$ and $\nabla S$ is of weight $k$, if the total number of derivatives appearing in that term is $k$. We say that a polynomial in $|\xi'|$ and $\nabla S$ is of weight $k$, if each term of the polynomial is of weight $k$. 

\begin{lem}\label{locallem}
For $j\geq 1$, the coefficient $c_{j}$ is an integral over $\partial X$ of a polynomial of products of derivatives in the metric and $S$, with weight $j-1$.
\end{lem}

Indeed, from \eqref{lowerorder}, using $S(x')=0$ and $\grad S(x')=0$ at $x'=0$, the terms of the symbols of $\mathbb{A}_{++}$ and $\mathbb{A}_{+-}$ are a sum of products of second $\frac{\partial^2S}{\partial x_i \partial x_j}$ and higher-order derivatives of $S$, and powers of $(R^2+|\xi'|^2)$. The structure is preserved by \eqref{recursionrelation} and \eqref{fullsymb}. One confirms that at $(\xi, R)=(0,1)$ they have the above-mentioned weight. 

Using Gilkey's invariant theory \cite{gilkey}, the local formulas from Lemma \ref{locallem} may be interpreted geometrically in terms of curvatures. Similar arguments can be found also in \cite{sherpolte}. The situation is simplified by the fact that the curvature tensor $R_{ijkl}$ of $\mathbb{R}^n$ vanishes. We let $L_{ij}$ denote the second fundamental form and $\nabla_{\partial X}$ the covariant derivatives on $\partial X$. 

\begin{lem}
\label{globallem}
The coefficient $c_j$ is an integral over $\partial X$ of a polynomial in the entries of $\nabla_{\partial X}^k L$, $0 \leq k\leq j-2$. The total number of covariant derivatives appearing in each term of the polynomial is $j-2$. 
\end{lem}

\large
\section{Meromorphic extension of the magnitude}
\normalsize

The magnitude of a domain in $\R^n$ can be expressed in terms of its volume and its Dirichlet-to-Neumann operator (associated with \eqref{thebvpgener}). We shall see that the operator $R\mathbb{A}(R)$ admits a holomorphic extension to $\C$. Using the theory of meromorphic families of Fredholm operators, we deduce that also the Dirichlet-to-Neumann operator $\Lambda(R)$, and therefore the magnitude function, extends meromorphically to $\C$. We discuss the pole structure of $\mathcal{M}_X$ in the context of scattering theory in Subsection \ref{scattpolesubs}.

\subsection{Meromorphic extension of $\mathcal{M}_X$}
\label{meroextsubs}

\begin{lem}
\label{thickaholo}
The function $\C\ni R\mapsto R\mathbb{A}(R)\in \mathbb{B}(\mathcal{H})$ is holomorphic in the norm topology.
\end{lem}

\begin{proof}
We decompose $\Omega$ as in Lemma \ref{exteningtoomega}. That is, we decompose $\overline{\Omega}=\overline{\Omega'}{\cup}\overline{\Omega''}$ where $\Omega'$ is a bounded domain with smooth boundary and $\Omega''$ is a smooth domain with boundary, compact complement and $\mathrm{d}(X,\Omega'')=r>0$. Let $\tilde{\mathcal{A}}(R)\in \mathbb{B}(H^{1/2}(\partial X),H^{2m}(\Omega'))$ denote the composition of $R\mathcal{A}_n(R)$ and the mapping $H^{2m}(\Omega)\to H^{2m}(\Omega')$ induced by restriction. By construction, $\tilde{\mathcal{A}}(R)$ is well defined for all $R\in \C$. Indeed, for $f\in H^{1/2}(\partial X)$, $\tilde{\mathcal{A}}(R)f$ is given by 
$$\tilde{\mathcal{A}}(R)f(x):=\kappa_n\int_{\partial X} \mathrm{e}^{-R|x-y|}f(y)\mathrm{d}S(y), \quad x\in \Omega'.$$
Since $\Omega'$ is compact, it follows that $R\mapsto \tilde{\mathcal{A}}(R)\in \mathbb{B}(H^{1/2}(\partial X),H^{2m}(\Omega'))$ depends holomorphically on $R\in \C$. The entries $(RA_{jl}(R))_{j,l=0}^n$ are obtained from $\tilde{\mathcal{A}}(R)$ using holomorphic operations, and holomorphicity of $R\mathbb{A}(R)$ follows. 
\end{proof}

The following version of the analytic Fredholm theorem for holomorphic families of Fredholm  operators is well known. Its proof can be found in \cite[Proposition 1.1.8]{leschhab}.

\begin{lem}
\label{merofamily}
Let $\mathcal{H}_1$ and $\mathcal{H}_2$ be Hilbert spaces, $D\subseteq \C$ a connected domain and $T:D\to \mathrm{Fred}(\mathcal{H}_1,\mathcal{H}_2)$ a holomorphic function to the open subset $\mathrm{Fred}(\mathcal{H}_1,\mathcal{H}_2)\subseteq \mathbb{B}(\mathcal{H}_1,\mathcal{H}_2)$ of Fredholm operators. Assume that there exists a $\lambda_0\in D$ such that $T(\lambda_0)$ is invertible. Then the set 
$$Z:=\{z\in D: 0\in \mathrm{Spec}(T(z))\}$$
is discrete in $D$, and $\lambda\mapsto T(\lambda)^{-1}$ exists as a meromorphic function in $D$. Moreover, near any $z\in Z$, $T^{-1}$ has a pointwise norm-convergent Laurent expansion 
\begin{equation}
\label{laurentexp}
T(\lambda)^{-1}=\sum_{k=-N}^\infty T_k(\lambda-z)^k,
\end{equation}
where $T_k$ are finite rank operators for $k<0$.
\end{lem}

We now apply Lemma \ref{merofamily} to the Dirichlet-to-Neumann operator. By Theorem \ref{lambdaisps}, $\Lambda(R)=\mathbb{A}_{+-}(R)^{-1}(1-\mathbb{A}_{++}(R))$ for $\mathrm{Re}(R)$ large enough. Indeed, by Corollary \ref{invertingaplumi}, $\mathbb{A}_{+-}(R)$ is invertible for $\mathrm{Re}(R)$ large enough. It now follows from Lemma \ref{thickaholo} and Lemma \ref{merofamily} that $\mathbb{A}_{+-}(R)^{-1}$ defines a meromorphic function in $\C$. Since $R(1-\mathbb{A}_{++}(R))$ is holomorphic by Lemma \ref{thickaholo}, we  have shown the following theorem. 

\begin{thm}
\label{dtonmero}
The operator $R\mapsto \Lambda(R)=\mathbb{A}_{+-}(R)^{-1} (1-\mathbb{A}_{++}(R))\in \mathbb{B}(\mathcal{H}_+,\mathcal{H}_-)$ admits a meromorphic continuation from $\Gamma_{\frac{\pi}{n+1}}$ to $\C$. 
\end{thm}

\begin{remark}
From elliptic regularity on the closed manifold $\partial X$, Lemma \ref{calderprojlemacomp} and Lemma \ref{thickaholo} one concludes that $\Lambda(R)$ admits an expansion near $R_0\in \C\setminus \{0\}$ as in Equation \eqref{laurentexp}, where $T_k$ are smoothing operators for $k<0$.
\end{remark}

\begin{deef}
\label{scattpo}
The set of poles of $\Lambda(R)$ is denoted by $\mathrm{P}_\Lambda(X)$. 
\end{deef}

\begin{prop}
For any $\alpha\in [0,\pi/2)$, there is a constant $C_{\alpha,X}\geq 0$, such that 
$$\mathrm{P}_\Lambda(X)\subseteq \{R\in \C\setminus \Gamma_{\frac{\pi}{n+1}}: \mathrm{Re}(R)< C_{\alpha,X} \;\mbox{if $R\in \Gamma_\alpha$}\}.$$
\end{prop}

\begin{proof}
If $R\in \mathrm{P}_\Lambda(X)\cap \Gamma_{\frac{\pi}{n+1}}$, then $\Lambda(R)$ does not exist, while the Barcel\'o-Carbery boundary value problem \eqref{thebvpgener} has a unique weak solution by Proposition \ref{uniquesolu}. This is a contradiction. The proof is complete upon noting that $\mathrm{P}_\Lambda(X) \subset \{R \in \C : R=0 \mbox{   or    }\mathbb{A}_{+-}(R) \text{ not invertible}\}$ and $\mathbb{A}_{+-}(R)$ is invertible for large $\mathrm{Re}(R)$ by Corollary \ref{invertingaplumi}.
\end{proof}

\begin{thm}
The function $\mathcal{M}_X(R):=\textnormal{mag}(R\cdot X)$ has a meromorphic continuation from $R> 0$ to $R\in \C$ with poles being of at most finite order. Its set of poles $\mathrm{P}_{\textnormal{Mag}}(X)$ is invariant under complex conjugation and for any $\alpha\in [0,\pi/2)$, there is a $C_{X,\alpha}\geq 0$ such that $\mathrm{P}_{\textnormal{Mag}}(X)$ is contained in $\mathrm{P}_{\Lambda}(X)\setminus\{0\}\subset \{R\in \C\setminus \Gamma_{\frac{\pi}{n+1}}: \mathrm{Re}(R)< C_{\alpha,X} \;\mbox{if $R\in \Gamma_\alpha$}\}$.
\end{thm}

\begin{proof}
We define $\mathrm{e}_j\in \mathcal{H}$ as the $j$-th basis vector, defining a constant function on $\partial X$. By Equation \eqref{expandinmagindn}, 
$$\mathcal{M}_X(R)=\frac{\textnormal{vol}_n(X)}{n!\omega_n}R^n+\frac{1}{n!\omega_n}\sum_{\frac{m}{2}<j\leq m}\sum_{0\leq l<m/2} R^{n-2j+2l}\langle \Lambda(R)\mathrm{e}_{2l},\mathrm{e}_{2j-1}\rangle_{L^2(\partial X,\C^{2m})}.$$
The right hand side depends meromorphically on $R\in \C$ by Theorem \ref{dtonmero}, and its poles are of finite order by Lemma \ref{thickaholo} and \ref{merofamily}. It is holomorphic outside $\mathrm{P}_\Lambda(X)\subseteq \{R\in \C\setminus \Gamma_{\frac{\pi}{n+1}}: \mathrm{Re}(R)< C_{\alpha,X} \;\mbox{if $R\in \Gamma_\alpha$}\}$.

From the meromorphic continuation and $\mathrm{lim}_{R\to 0^+} \mathcal{M}_X(R) = 1$ (see \cite[Theorem 1]{barcarbs}), one sees that $\mathcal{M}_X$ is holomorphic near $R=0$.

For the invariance under complex conjugation, note that the identity $\overline{\Lambda(R)} = \Lambda(\overline{R})$ follows from the Barcel\'o-Carbery problem \eqref{thebvpgener} for $R \in \Gamma_{\frac{\pi}{n+1}}$. Hence, $\overline{\Lambda(R)} = \Lambda(\overline{R})$ holds for all $R \not \in \mathrm{P}_{\Lambda}(X)$. We conclude $\overline{\mathcal{M}_X(R)} = \mathcal{M}_X(\overline{R})$ for all $R \not \in \mathrm{P}_{\Lambda}(X)$, and therefore $\overline{\mathrm{P}_{\textnormal{Mag}}(X)} = \mathrm{P}_{\textnormal{Mag}}(X)$.
\end{proof} 

Note that for $R \in \C_+$, the meromorphic extension $\mathcal{M}_X(R)$ may be computed from the boundary problem \eqref{thebvpgener}, or equivalently its reformulation \eqref{expandinmagindn} in terms of the Dirichlet-to-Neumann operator.

\begin{figure}
   \centering
   \subfigure[]{
   \includegraphics[height=5.35cm]{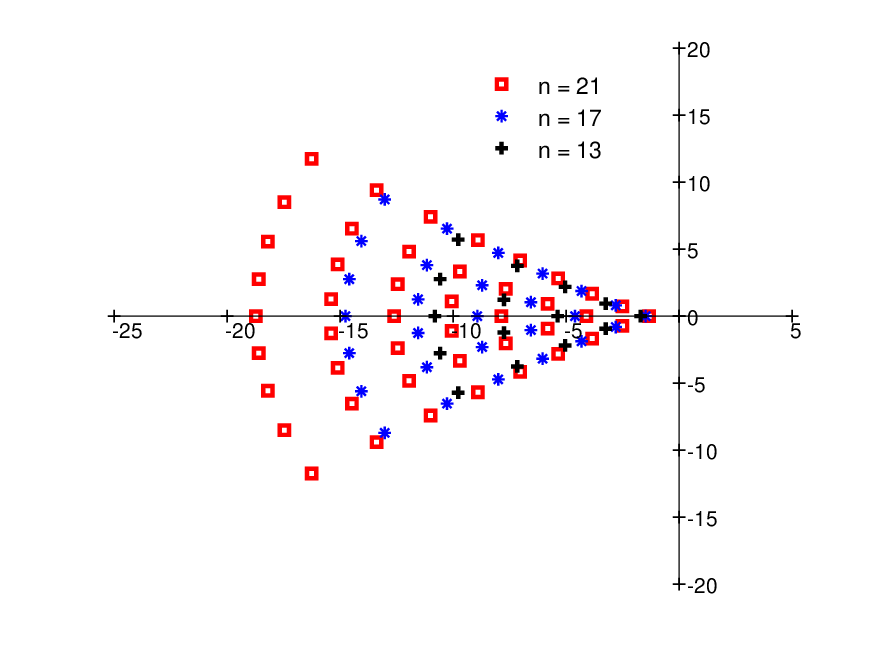}}
 \subfigure[]{
  \includegraphics[height=5.35cm]{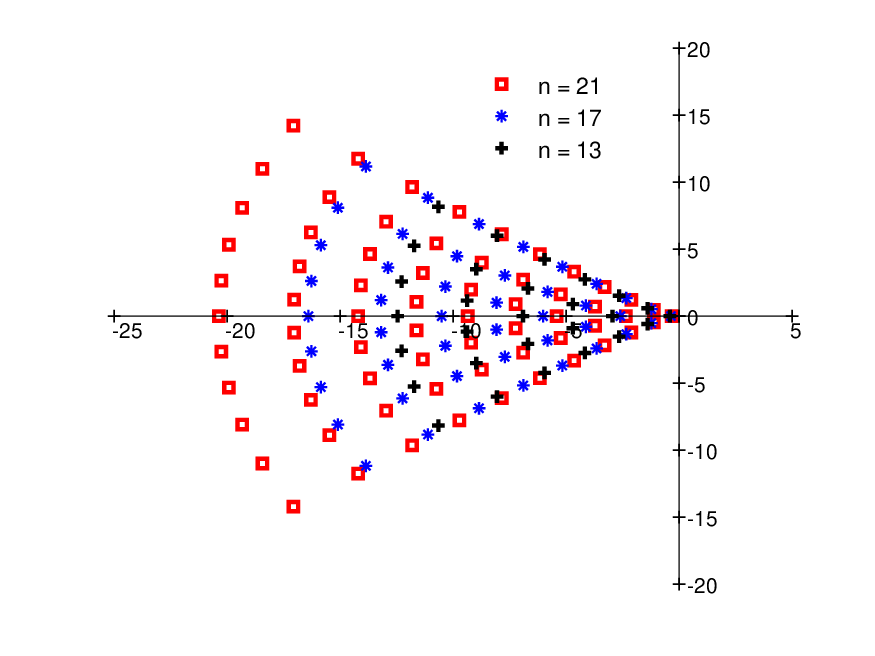}}
  \caption{The set of (a) poles and (b) zeros of $\mathcal{M}_{B_n}$ in dimensions $n = 13, 17, 21$.}
\label{plot}
\end{figure}

\begin{example}
\label{examplecompact}
Let $\Omega = \R \setminus X$ be disconnected with a bounded component $\Omega_0$. The fact that $L^2(\Omega_0)$ is spanned by eigenfunctions of the Dirichlet Laplacian on $\Omega_0$ shows that the boundary value problem \eqref{thebvpgener} lacks a unique weak solution for up to infinitely many $R$ on the imaginary axis. The eigenvalues of the Dirichlet Laplacian on $\Omega_0$ form an infinite set of poles of $\Lambda(R)$ and $\mathcal{M}_X$ on the imaginary axis $i \R$. Also for generic domains $X$, even convex ones, the construction of the Dirichlet-to-Neumann operator indicates that it would likely have an infinite set of poles $\mathrm{P}_{\Lambda}(X)$.

\end{example}

\begin{example}
\label{exampleBn}
In the special case of the unit ball $X=B_n$, using representation theory or Willerton's formulas \cite{will} one can show that $\mathcal{M}_{B_n}$ is a rational function. The number of poles is at most $\frac{(n-1)(n-3)}{8}$, while the number of zeros is bounded by $\frac{(n+3)(n+1)}{8}$. As an example of the structure apparent in $\mathcal{M}_{B_n}$, based on the formulas in \cite{will}  and high-precision numerical calculations, Figure \ref{plot} depicts the set of poles $\mathrm{P}_{\textnormal{Mag}}(B_n)$ as well as the set of zeros for dimensions $n = 13, 17, 21$.  
\end{example}

\begin{example}
\label{exampleShell}
For $X=(2B_3)\setminus B_3^\circ$ -- a spherical shell in $\R^3$ -- the magnitude function $\mathcal{M}_X$ is not rational and contains an infinite sequence of poles accumulating on the curve $\mathrm{Re}(R)= \log(|\mathrm{Im}(R)|)$, see Figure \ref{shellfigure}. Indeed, the magnitude function of $X$ is given by
$$\mathcal{M}_X(R)=\mathcal{M}_{B_3}(2R)-\frac{1}{3! \omega_3} \mathrm{vol}_3(B_3)\ R^3-\frac{1}{3! \omega_3 R}\int_{S^2}\mathcal{D}^{3}_R h_R\,\rd S,$$
where $h_R$ solves the Barcel\'{o}-Carbery problem in the ball $B_3$. Using spherical symmetry, $h_R(x)=v_R(|x|)$ where $v_R$ solves $(R^2-\partial_r^2-2r^{-1}\partial_r)^2v_R=0$ in $[0,1]$ and satisfies the boundary conditions $v_R(1)=1$ and $v_R'(1)=0$. One can compute $\mathcal{D}^{3}_R h_R=v_R'''(1)+2v_R''(1)$ and obtain the formula
$$ \mathcal{M}_X(R)=\frac{7}{6}R^3+5R^2+2R+2+ \frac{\mathrm{e}^{-2R}(R^2+1)+2R^3-3R^2+2R-1}{\sinh(2R)-2R}.$$
\end{example}

\begin{figure}
   \centering
   \includegraphics[height=5.35cm]{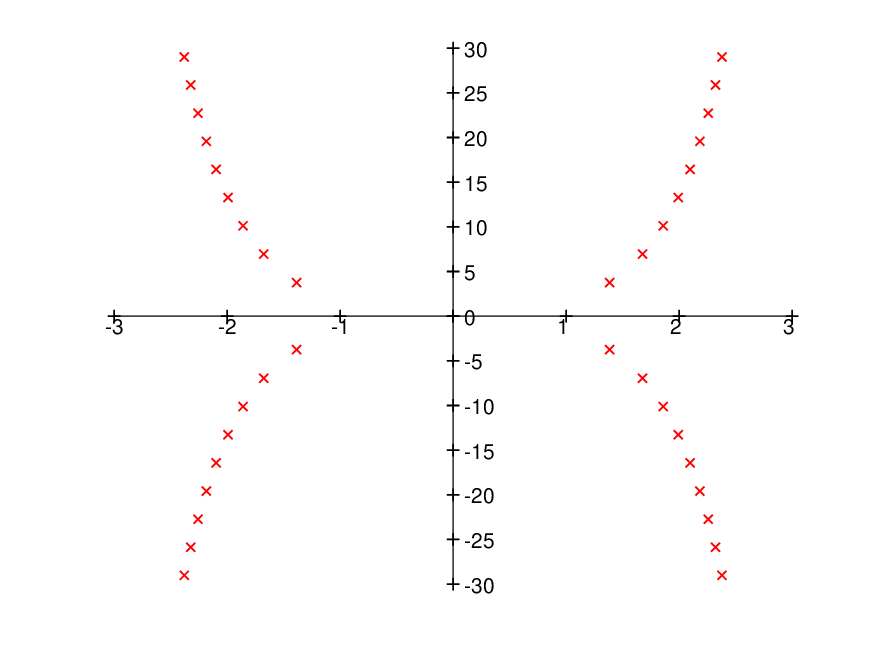}
  \caption{The set of poles  $\mathcal{M}_{X}$ for the spherical shell $X=(2B_3)\setminus B_3^\circ$.}
\label{shellfigure}
\end{figure}

\subsection{Scattering resonances and the magnitude}
\label{scattpolesubs}

The results of Subsection \ref{meroextsubs} may be interpreted in the context of scattering theory. There one studies the  exterior problem of second order
\begin{equation*}
(R^2-\Delta) u=0\quad\mbox{in $\Omega$} ,
\end{equation*}
with Dirichlet or Neumann boundary conditions at $\partial\Omega = \partial X$ \cite{melrose}. Similar to the setting in this article, the Dirichlet-to-Neumann operator extends meromorphically to $\C$, where the poles with negative real part are called \emph{scattering resonances}. In this way, the points in $\mathrm{P}_{\Lambda}(X)$ with negative real part are analogous to the scattering resonances for the Barcel\'o-Carbery boundary problem \eqref{thebvpgener} of order $2m$. In addition, the Barcel\'o-Carbery boundary problem allows at most a finite number of poles with positive real part in any sector $\Gamma \subset \C_+$.

Our approach therefore provides intuition and techniques to study the meromorphic function $\mathcal{M}_X$ based on scattering theory for a parameter dependent higher-order operator. A typical result shows that for a generic domain $X$ there are infinitely many scattering resonances, which intersect any sector $\Lambda \supset \C_+$ of opening angle $< \pi$. Indeed, for convex $X$, it is known that the number of scattering resonances in $\{R \in \C: -\mathrm{Re}(R)\lesssim|\mathrm{Im}(R)|^c\}$ is finite for a $0<c<1$, while for non-convex $X$ scattering resonances may be logarithmically close to $i \R$ \cite{melrose}, { as is the case for the magnitude function of the shell domain in Example \ref{exampleShell}}. This gives an indication beyond Section \ref{meroextsubs} that $\mathrm{P}_{\Lambda}(X)$, and possibly also $\mathrm{P}_{\textnormal{Mag}}(X)$, would be infinite for a large class of domains.

\begin{appendix}
\large
\section{Pseudodifferential operators with parameters} 
\label{apppseudo}
\normalsize

In this appendix we give a brief introduction to pseudo-differential operators with parameters. We also prove Theorem \ref{theboundaryops} below in Subsection \ref{theorem9} and Lemma \ref{condexplem} in Subsection \ref{lemma16}. 

\subsection{Brief introduction to pseudo-differential operators with parameters}
\label{briefintropara}

The theory of pseudo-differential operators with parameters can be developed in the same way as the standard theory for pseudo-differential operators by treating the parameter as an additional covariable. We refer the reader to \cite[Chapter 2]{GrubbGreenBook} or \cite[Chapter 1.7]{gilkeyinvariance} for more details, and \cite[Chapter XVIII]{hormanderIII} for the theory of pseudo-differential operators.

We fix a closed $n$-dimensional manifold $M$ and a sector $\Gamma\subseteq \C$. Following the introduction in the paper, we denote the parameter by $R\in \Gamma$. The procedure of defining pseudo-differential operators from symbols in local charts is a well known procedure, see for instance \cite[Page 83-86]{hormanderIII} or \cite[Chapter 8]{GrubbDistOp}, and we leave out this aspect working directly on a manifold using local charts. We write $(x,\xi)\in T^*M$ for an element of the cotangent bundle, and by an abuse of notation we often implicitly identify $\xi\in T_x^*M$ with a covector in $\R^n$ via an unspecified local chart. We only treat the case of scalar operators, the matrix-valued case works analogously.

\begin{deef}
A \emph{symbol with parameter} of order $s$ is a function $a\in C^\infty(T^*M\times \Gamma)$ such that in every local chart $U\subseteq M$ the following holds: For any $\alpha,\beta\in \N^n$ and $k\in \N$, there is a constant $C=C(\alpha,\beta,k)$ such that in local coordinates
$$|\partial_x^\alpha\partial_\xi^\beta\partial_R^ka(x,\xi,R)|\leq C(1+|\xi|^2+|R|^2)^{(s-|\beta|-k)/2}, \quad (x,\xi,R)\in T^*U\times \Gamma.$$
The space of symbols of order $s$ forms a Fr\'echet space $S^s(M;\Gamma)$. 
\end{deef}

Note that $S^{-\infty}(M;\Gamma):=\cap_{s\in \R} S^s(M;\Gamma)=\mathcal{S}(T^*M\times \Gamma)$ -- the Schwartz space of $T^*M\times \Gamma$. We write $a_1\sim a_2$ if $a_1-a_2\in S^{-\infty}(M;\Gamma)$. The space of symbols $S^*(M;\Gamma)$ is asymptotically complete in the sense that when $a_{s_j}\in S^{s_j}(M;\Gamma)$, $j\in \N$, and $s_j\to -\infty$ the sum $\sum_j a_{s_j}\in S^s(M;\Gamma)$, for $s=\max_j s_j$, can be defined uniquely modulo $S^{-\infty}(M;\Gamma)$ (see \cite[Proposition 18.1.3]{hormanderIII}). 

The operator $Op(a)$ defined from $a\in S^s(M;\Gamma)$ is a family of operators $(Op(a)_R:C^\infty(M)\to C^\infty(M))_{R\in \Gamma}$. It is defined in local coordinates on a function $f$ with Fourier transform $\hat{f}$ by 
$$Op(a)_Rf(x):=\int_{\R^n} \e^{ix\xi}a(x,\xi,R)\hat{f}(\xi)\mathrm{d}\xi.$$
See \cite[Theorem 18.1.6]{hormanderIII}. We remark that $Op(a)$ depends on choices of local coordinates. The discussion on \cite[Page 85]{hormanderIII} implies that $Op(a)$ is uniquely determined by $a$ modulo smoothing operators with parameter, i.e.~ modulo the space $\Psi^{-\infty}(M;\Gamma)$ of families of operators with parameter dependent Schwartz kernel in $\mathcal{S}(\Gamma; C^\infty(M\times M))$. By \cite[Theorem 18.1.13]{hormanderIII}, for $a\in S^s(M;\Gamma)$, $Op(a)$ defines a family of operators on Sobolev spaces $H^t(M)\to H^{t-s}(M)$ for any $t\in \R$.

We say that a symbol $a\in S^s(M;\Gamma)$ is \emph{homogeneous of degree $s$} if there is a $c$ such that $a(x,t\xi,tR)=t^sa(x,\xi,R)$ for $|\xi|^2+|R|^2, t\geq c$. We let $S^s_{\rm cl}(M;\Gamma)$ denote the asymptotic completion of the space of all homogeneous operators of order $s-j$, where $j$ runs over all natural numbers. A \emph{classical symbol with parameter} of order $s$ is a symbol $a\in S^s_{\rm cl}(M;\Gamma)$. In particular, $a\in S^s_{\rm cl}(M;\Gamma)$ if and only if there are uniquely determined symbols $\sigma_{s-j}(a)$ homogeneous of degree $s$ with $a\sim \sum_j \sigma_{s-j}(a)$. 

A \emph{pseudo-differential operator with parameter} of order $s$ is an operator of the form $A=Op(a)+ T$ where $a\in S^s_{\rm cl}(M;\Gamma)$ and $T\in \Psi^{-\infty}(M;\Gamma)$. As above, $a\in S^s_{\rm cl}(M;\Gamma)$ is uniquely determined by $A$ modulo $S^{-\infty}(M;\Gamma)$. We call $a$ a \emph{full symbol} of $A$. We let $\Psi^s(M;\Gamma)$ denote the space of pseudo-differential operators with parameter of order $s$. The product as operators on $C^\infty(M)$ defines a product 
$$\Psi^{s_1}(M;\Gamma)\times \Psi^{s_2}(M;\Gamma)\to \Psi^{s_1+s_2}(M;\Gamma).$$
By \cite[Theorem 18.1.8]{hormanderIII}, the \emph{product at the level of full symbols} $a_i\in S^{s_i}_{\rm cl}(M;\Gamma)$ of $A_i\in \Psi^{s_i}(M;\Gamma)$, $i=1,2$, is  determined up to $S^{-\infty}(M;\Gamma)$ by the asymptotic expansion
$$[a_1*a_2](x,\xi,R)=\sum_{\alpha\in \N^n} \frac{i^{|\alpha|}}{\alpha!}\partial_\xi^\alpha a_1(x,\xi,R)\partial_x^\alpha a_2(x,\xi,R).$$
The arguments in the matrix-valued case are analogous. If $A=(A_{j,k})_{j,k=1}^m$ and $A'=(A'_{j,k})_{j,k=1}^m$ are matrices of pseudo-differential operators with $A_{j,k}\in \Psi^{s_{j,k}}(M;\Gamma)$ and $A'_{j,k}\in \Psi^{s'_{j,k}}(M;\Gamma)$ for some $s_{j,k}, s'_{j,k}\in \R$, $j,k=1,\ldots, m$, then a full symbol $a''=(a''_{j,k})_{j,k=1}^m$ of $A'':=AA'=(\sum_{l=1}^m A_{j,l}A'_{l,k})$ is given by 
$$a''_{j,k}(x,\xi,R)=\sum_{l=1}^m\sum_{\alpha\in \N^n} \frac{i^{|\alpha|}}{\alpha!}\partial_\xi^\alpha a_{j,l}(x,\xi,R)\partial_x^\alpha a'_{l,k}(x,\xi,R),$$
where $a=(a_{j,k})_{j,k=1}^m$ and $a'=(a'_{j,k})_{j,k=1}^m$ are full symbols of $A$ and $A'$, respectively. In particular, $A''_{j,k}\in \Psi^{s_{j,k}''}(M;\Gamma)$ where $s''_{j,k}=\max\{s_{j,l}+s'_{l,k}: \, l=1,\ldots,m\}$.
 
Let $A\in \Psi^{s}(M;\Gamma)$ have a full symbol $a\in S^{s}_{\rm cl}(M;\Gamma)$ with symbol expansion $a\sim \sum_{j=0}^\infty \sigma_{s-j}(a)$. The function $\sigma_{s}(a)\in C^\infty(T^*M\times \Gamma)$ is called the \emph{principal symbol} of $A$. If there is a compact $K\subseteq T^*M\times \Gamma$ such that the principal symbol $\sigma_{s}(a)$ is invertible outside $K$, we say that $A$ is \emph{elliptic}. By a similar argument as in \cite[Theorem 18.1.24]{hormanderIII}, $A\in \Psi^{s}(M;\Gamma)$ is elliptic if and only if $A$ admits an operator $P\in \Psi^{-s}(M;\Gamma)$ such that $AP-1,PA-1\in \Psi^{-\infty}(M;\Gamma)$. The operator $P$ is called a \emph{parametrix} of $A$. 

For matrix-valued operators ellipticity is used in an analogous way. Fix $s$. Let $A=(A_{j,k})_{j,k=1}^m$ be a matrix of pseudo-differential operators with $A_{j,k}\in \Psi^{s+j-k}(M;\Gamma)$ whose full symbol $a_{j,k}$ has an asymptotic expansion $a_{j,k}\sim \sum_{l=0}^\infty \sigma_{s+j-k-l}(a_{j,k})$ where $\sigma_{s+j-k-l}(a_{j,k})$ is homogeneous of order $s+j-k-l$. We say that $A$ is elliptic if the homogeneous principal symbol $(\sigma_{s+j-k}(a_{j,k}))_{j,k=1}^m$ takes invertible values on $R^2+|\xi|^2=1$. In the matrix-valued case, $A$ is elliptic if and only if there is a matrix $P=(P_{j,k})_{j,k=1}^m$ of pseudo-differential operators with $P_{j,k}\in \Psi^{-s+j-k}(M;\Gamma)$ and $AP-1$ and $PA-1$ being matrices of pseudo-differential operators with entries in $\Psi^{-\infty}(M;\Gamma)$.

\subsection{Proof of Theorem \ref{theboundaryops}}
\label{theorem9}

To prove Theorem \ref{theboundaryops}, we need to show that $A_{jk}\in \Psi^{j-k}(\partial X;\Gamma)$ and compute $\sigma_{j-k}(a_{jk})(x',\xi',R)$ and $\sigma_{j-k-1}(a_{jk})(x',\xi',R)$ for $(x',\xi',R)\in T^*\partial X\times \Gamma$. We pick local coordinates at a point on $\partial X$. We can assume that this point is $0\in \R^n$ and that the coordinates are of the form $(x',S(x'))$, where $x'$ belongs to some neighborhood of $0\in \R^{n-1}$ and $S$ is a scalar function with $S(0)=0$ and $\grad S(0)=0$. In these coordinates, it is a well known fact that $(n-1)H(0)=\Delta S(0)$.  

We make use of the auxiliary distributions 
\[
b_{r,N}(u,z):=\int_{\R} \frac{w^r\mathrm{e}^{izw}}{(u+w^2)^N}\mathrm{d}w.
\]
Here $r\in \N$, $N\in \Z$, $u>0$ and $z\in \R$.

\begin{prop}
\label{brncomp}
In a distributional sense, we can write $b_{r,N}$ in the following way:
\[
b_{r,N}(u,z)=
\begin{cases} 
(-i\partial_z)^r(u-\partial_z^2)^{-N}\delta_{z=0}, \;& N\leq 0,\\
{}\\
(-i\partial_z)^r\sum_{k=0}^{N-1}\tilde{c}_{k,r,N} \frac{|z|^k\mathrm{e}^{-|z|\sqrt{u}}}{u^{N-(k+1)/2}}, \;& N> 0,
\end{cases}
\]
for some coefficients $\tilde{c}_{k,r,N}$. 
\end{prop}

\begin{proof}
Basic manipulations show that  $b_{r,N}=(-i\partial_z)^rb_{0,N}$ so we may assume that $r=0$. If $N\leq 0$, we can in a distributional sense compute 
\begin{align*}
b_{0,N}(u,z)=\int_{\R} \frac{\mathrm{e}^{izw}}{(u+w^2)^N}\mathrm{d}w=\int_{\R} (u+w^2)^{-N}\mathrm{e}^{izw}\mathrm{d}w=(u-\partial_z^2)^{-N}\int_{\R}\mathrm{e}^{izw}\mathrm{d}w=(u-\partial_z^2)^{-N}\delta_{z=0},
\end{align*}
and the Proposition holds for $N\leq 0$.

If $N>0$, we can compute that 
$$b_{0,N}(u,z)=\int_{\R} \frac{\mathrm{e}^{izw}}{(u+w^2)^N}\mathrm{d}w=\frac{(-1)^{N-1}}{(N-1)!}\frac{\partial^{N-1}}{\partial u^{N-1}}\int_{\R} \frac{\mathrm{e}^{izw}}{u+w^2}\mathrm{d}w=\frac{(-1)^{N-1}}{(N-1)!}\frac{\partial^{N-1}}{\partial u^{N-1}}b_{0,1}(u,z).$$
Therefore, it remains to prove that $b_{0,1}(u,z)=\tilde{c}\frac{\mathrm{e}^{-|z|\sqrt{u}}}{\sqrt{u}}$ for some constant $\tilde{c}$. A change of coordinate shows that if $u\neq 0$ then 
$$b_{0,1}(u,z)=\frac{1}{\sqrt{u}}b_{0,1}(1,z\sqrt{u}).$$
We are left to show that $b_{0,1}(1,z)= \tilde{c}\mathrm{e}^{-|z|}$ for some constant $\tilde{c}$. By definition, $b_{0,1}(1,z)$ is the Fourier transform of the function $w\mapsto (1+w^2)^{-1}$ and we conclude $b_{0,1}(1,z)= \tilde{c}\mathrm{e}^{-|z|}$ (for some constant $\tilde{c}$) from a standard computation of Fourier transforms. We conclude the case $N>0$.

\end{proof}

By partially Fourier transforming in the normal direction and taking an exterior limit, we see that in this choice of coordinates a full symbol of $A_{jk}$  is given by 
\begin{align*}
a_{jk}(x',y',\xi',R)=&b_{0,m-p-q}(R^2+|\xi'|^2,S(x')-S(y')), \\ 
&\quad\quad\mbox{when  $j=2p, k=n-2q$}\\
a_{jk}(x',y',\xi',R)=&b_{1,m-p-q}(R^2+|\xi'|^2,S(x')-S(y'))+\\
&(\xi'\cdot \grad S(x'))b_{0,m-p-q}(R^2+|\xi'|^2,S(x')-S(y')), \\ 
&\quad\quad\mbox{when  $j=2p+1, k=n-2q$}\\
a_{jk}(x',y',\xi',R)=&b_{1,m-p-q}(R^2+|\xi'|^2,S(x')-S(y'))+\\
&(\xi'\cdot \grad S(y'))b_{0,m-p-q}(R^2+|\xi'|^2,S(x')-S(y')),\\ 
&\quad\quad\mbox{when  $j=2p, k=n-2q-1$}
\end{align*}
\begin{align*}
a_{jk}(x',y',\xi',R)=&b_{2,m-p-q}(R^2+|\xi'|^2,S(x')-S(y'))+\\
&((\xi'\cdot \grad S(y'))+(\xi'\cdot \grad S(x')))b_{1,m-p-q}(R^2+|\xi'|^2,S(x')-S(y'))+\\
&(\xi'\cdot \grad S(x'))(\xi'\cdot \grad S(y'))b_{0,m-p-q}(R^2+|\xi'|^2,S(x')-S(y')), \\ 
&\quad\quad\mbox{when  $j=2p+1, k=n-2q-1$}.
\end{align*} 
It follows from Proposition \ref{brncomp} that $a_{j,k}\in S^{j-k}_{\rm cl}(\partial X;\Gamma)$ and that $A_{jk}\in \Psi^{j-k}(\partial X;\Gamma)$. The principal symbol is obtained by setting $x'=y'$. Hence, there are coefficients $c_{j,k}$ (the coefficients of the leading order contributions in $u$ from Proposition \ref{brncomp}) independent of $\partial X$ such that at $x'=0$
\begin{align*}
\sigma_{j-k}(a_{jk})(0,\xi',R)=&b_{0,m-p-q}(R^2+|\xi'|^2,0)=c_{j,k}(R^2+|\xi'|^2)^{(j-k)/2}+\mbox{l.o.t.}, \\ 
&\quad\mbox{when  $j=2p, k=n-2q$}\\
\sigma_{j-k}(a_{jk})(0,\xi',R)=&b_{1,m-p-q}(R^2+|\xi'|^2,0))=c_{j,k}(R^2+|\xi'|^2)^{(j-k)/2}+\mbox{l.o.t.}, \\ 
&\quad\quad\mbox{when  $j=2p+1, k=n-2q$}\\
\sigma_{j-k}(a_{jk})(0,\xi',R)=&b_{1,m-p-q}(R^2+|\xi'|^2,0)=c_{j,k}(R^2+|\xi'|^2)^{(j-k)/2}+\mbox{l.o.t.},\\ 
&\quad\quad\mbox{when  $j=2p, k=n-2q-1$}\\
\sigma_{j-k}(a_{jk})(0,\xi',R)=&b_{2,m-p-q}(R^2+|\xi'|^2,0)=c_{j,k}(R^2+|\xi'|^2)^{(j-k)/2}+\mbox{l.o.t.}, \\ 
&\quad\quad\mbox{when  $j=2p+1, k=n-2q-1$}.
\end{align*}
Here we are using that $\grad S(0)=0$. This proves Theorem \ref{theboundaryops}.a.

The part of the symbol of order $j-k-1$, denoted by $\sigma_{j-k-1}(a_{jk})(x',\xi',R)$, is at $x'=0$ given by the term of order $j-k-1$ in the expression $a_{jk}(0,0,\xi',R)-i\sum_{l=1}^{n-1}\frac{\partial^2 a_{jk}}{\partial \xi_l\partial y_l}(0,0,\xi',R)$. See more in for instance \cite[Theorem 7.13]{GrubbDistOp}. Using the fact that $\grad S(x')$ vanishes at $x'=0$, a computation as above shows that there are constants $\tilde{C}_{j,k-1}$ independent of the geometry such that
$$\sigma_{j-k-1}(a_{jk})(0,\xi',R)=\tilde{C}_{j,k-1}\sum_{l=1}^{n-1} \frac{\partial^2 S}{\partial x_l^2}(0)(R^2+|\xi'|^2)^{(j-k-1)/2}=C_{j,k-1}H(0)(R^2+|\xi'|^2)^{(j-k-1)/2},$$
where $C_{j,k-1}=(n-1)\tilde{C}_{j,k,-1}$. This proves Theorem \ref{theboundaryops}.b.

General lower order terms $\sigma_{j-k-l}(a_{jk})(x',\xi',R)$ of the symbol are given by the term of order $j-k-l$ in the expression
\begin{equation}\label{lowerorder}\sum_{\alpha \in \N_0^{n-1}} \frac{(-i)^{|\alpha|}}{\alpha!} \partial_\xi^\alpha \partial_x^\alpha a_{jk}(0,0,\xi',R)\ .\end{equation}

\subsection{Proof of Lemma \ref{condexplem}}
\label{lemma16}
We may reduce to the case where $A$ is of the form $A=Op(a)$ for a compactly based symbol $a\in S^m_{\rm cl}(U;\Gamma)$ for an open subset $U\subseteq \R^n$. Recall that compactly based means that $\chi a=a$ for a $\chi\in C^\infty_c(U)$. In this case we can consider $A$ as an operator $C^\infty(\R^n)\to \mathcal{E}'(\R^n)$, and we then have that 
$$\langle A1,1\rangle_{L^2(\R^n)}=\int_U\int_{\R^n}a(x,\xi,R)\delta_{\xi=0}\mathrm{e}^{ix\cdot\xi}\mathrm{d}\xi\mathrm{d}x=\int_Ua(x,0,R)\mathrm{d}x.$$
This proves Lemma \ref{condexplem}. We observe that $\langle A1,1\rangle_{L^2(\R^n)}=\int_Ua(x,0,R)\mathrm{d}x$ only holds for operators of the form $A=Op(a)$; the reduction to this case and the comparison of the global expression on $M$ to that on $\R^n$  gives rise to terms of order $\mathcal{O}(R^{-\infty})$.
\end{appendix}

\section*{Acknowledgments}

\noindent We thank Anthony Carbery and Gerd Grubb for fruitful discussions. We are grateful to Peter Gilkey, Tom Leinster, Mark Meckes and Simon Willerton for several interesting comments during the writing of this paper, and to Simon Willerton for sharing his results in \cite{will}. H.~G.~acknowledges partial support by ERC Advanced Grant HARG 268105. M.~G.~was supported by the Swedish Research Council Grant 2015-00137 and Marie Sklodowska Curie Actions, Cofund, Project INCA 600398. We  acknowledge additional travel support by the International Centre for Mathematical Sciences (Edinburgh). \\

\end{document}